\documentclass[a4paper]{amsart}
\usepackage[utf8]{inputenc}

\usepackage{hyperref} \hypersetup{%
  colorlinks,%
  linkcolor={red!80!black},%
  citecolor={blue!80!black},%
  urlcolor={blue!80!black}%
}

\usepackage{%
  amsthm,%
  amssymb,%
  amsmath,%
  mathtools,%
  stmaryrd,%
  enumitem,%
  eucal,%
}

\usepackage{tikz-cd}%
\usetikzlibrary{calc}%

\usepackage{wasysym}
\makeatletter%
\DeclareFontFamily{OMX}{MnSymbolE}{}%
\DeclareSymbolFont{MnLargeSymbols}{OMX}{MnSymbolE}{m}{n}%
\SetSymbolFont{MnLargeSymbols}{bold}{OMX}{MnSymbolE}{b}{n}%
\DeclareFontShape{OMX}{MnSymbolE}{m}{n}{%
  <-6>  MnSymbolE5%
  <6-7>  MnSymbolE6%
  <7-8>  MnSymbolE7%
  <8-9>  MnSymbolE8%
  <9-10> MnSymbolE9%
  <10-12> MnSymbolE10%
  <12->   MnSymbolE12%
}{}%
\DeclareFontShape{OMX}{MnSymbolE}{b}{n}{%
  <-6>  MnSymbolE-Bold5%
  <6-7>  MnSymbolE-Bold6%
  <7-8>  MnSymbolE-Bold7%
  <8-9>  MnSymbolE-Bold8%
  <9-10> MnSymbolE-Bold9%
  <10-12> MnSymbolE-Bold10%
  <12->   MnSymbolE-Bold12%
}{}%
\let\llangle\@undefined%
\let\rrangle\@undefined%
\DeclareMathDelimiter{\llangle}{\mathopen}%%
{MnLargeSymbols}{'164}{MnLargeSymbols}{'164}%
\DeclareMathDelimiter{\rrangle}{\mathclose}%%
{MnLargeSymbols}{'171}{MnLargeSymbols}{'171}%
\makeatother%

\usepackage{xparse}%

\newcommand{\op}{\mathrm{op}}

\newcommand{\CC}{\mathbb{C}}
\newcommand{\ZZ}{\mathbb{Z}}

\NewDocumentCommand{\FF}{O{p}}{\mathbb{F}_{#1}}

\newcommand{\category}[1]{\mathcal{#1}}
\newcommand{\A}{\category{A}}%
\newcommand{\C}{\category{C}}%
\newcommand{\T}{\category{T}}%

\NewDocumentCommand{\ob}{s m}{\operatorname{ob}\left(#2\right)}

\NewDocumentCommand{\set}{m o}{\left\{#1\right\IfValueT{#2}{\vert #2}\}}

\NewDocumentCommand{\Out}{s m}{\operatorname{Out}\IfBooleanTF{#1}{(#2)}{#2}}
\NewDocumentCommand{\Aut}{s m}{\operatorname{Aut}\IfBooleanTF{#1}{(#2)}{#2}}
\NewDocumentCommand{\Pic}{s m}{\operatorname{Pic}\IfBooleanTF{#1}{(#2)}{#2}}

\NewDocumentCommand{\Yoneda}{o o}{\mathbf{h}\IfValueT{#1}{_{#1}}}

\NewDocumentCommand{\add}{s m}{\operatorname{add}\IfBooleanTF{#1}{(#2)}{#2}}
\NewDocumentCommand{\smd}{s m}{\operatorname{smd}\IfBooleanTF{#1}{(#2)}{#2}}
\NewDocumentCommand{\thick}{s m}{\operatorname{thick}\IfBooleanTF{#1}{(#2)}{#2}}

\NewDocumentCommand{\proj}{s m}{\operatorname{proj}\IfBooleanTF{#1}{(#2)}{#2}}
\NewDocumentCommand{\mmod}{s m}{\operatorname{mod}\IfBooleanTF{#1}{(#2)}{#2}}
\NewDocumentCommand{\smmod}{s m}{\operatorname{\underline{mod}}\IfBooleanTF{#1}{(#2)}{#2}}
\NewDocumentCommand{\rad}{s m}{\operatorname{rad}\IfBooleanTF{#1}{(#2)}{#2}}
\NewDocumentCommand{\Mod}{s m}{\operatorname{Mod}\IfBooleanTF{#1}{(#2)}{#2}}
\NewDocumentCommand{\Spec}{s m}{\operatorname{Spec}\IfBooleanTF{#1}{(#2)}{#2}}

\NewDocumentCommand{\twBim}{O{1} m O{1}}{{}_{#1}{#2}_{#3}}

\NewDocumentCommand{\Gproj}{s m}{\operatorname{Gproj}\IfBooleanTF{#1}{(#2)}{#2}}
\NewDocumentCommand{\stableGproj}{s m}{\underline{\operatorname{Gproj}}\IfBooleanTF{#1}{(#2)}{#2}}
\NewDocumentCommand{\CM}{s m}{\operatorname{CM}\IfBooleanTF{#1}{(#2)}{#2}}
\NewDocumentCommand{\stableCM}{s m}{\underline{\operatorname{CM}}\IfBooleanTF{#1}{(#2)}{#2}}

\NewDocumentCommand{\dgSingCat}{o o m}{\operatorname{D}\IfValueTF{#1}{_{\mathrm{#1}}}{_{\mathrm{sg}}}\IfValueT{#2}{_{\mathrm{#2}}}(#3)_{\mathrm{dg}}}
\NewDocumentCommand{\DerCat}{o o m}{\operatorname{D}\IfValueT{#1}{^{\mathrm{#1}}}\IfValueT{#2}{_{\mathrm{#2}}}(#3)}
\NewDocumentCommand{\SingCat}{o o m}{\operatorname{D}\IfValueTF{#1}{_{\mathrm{#1}}}{_{\mathrm{sg}}}\IfValueT{#2}{_{\mathrm{#2}}}(#3)}
\NewDocumentCommand{\Ch}{o m}{\operatorname{C}\IfValueT{#1}{^{\mathrm{#1}}}(#2)}
\NewDocumentCommand{\KCh}{o
  m}{\operatorname{K}\IfValueT{#1}{^{\mathrm{#1}}}(#2)}
\NewDocumentCommand{\dgZ}{o m}{\operatorname{Z}\IfValueT{#1}{^{#1}}(#2)}
\NewDocumentCommand{\dgH}{O{\bullet} m}{H^{#1}(#2)}
\NewDocumentCommand{\dgCh}{o m}{\operatorname{C}_{\mathrm{dg}}\IfValueT{#1}{^{\mathrm{#1}}}(#2)}
\NewDocumentCommand{\AiMod}{s m}{A_\infty\operatorname{-Mod}\IfBooleanTF{#1}{(#2)}{#2}}
\NewDocumentCommand{\dgModdg}{s m}{\operatorname{dgMod}_{\mathrm{dg}}\IfBooleanTF{#1}{(#2)}{#2}}
\NewDocumentCommand{\dgMod}{s m}{\operatorname{dgMod}\IfBooleanTF{#1}{(#2)}{#2}}
\NewDocumentCommand{\coh}{s m}{\operatorname{coh}\IfBooleanTF{#1}{(#2)}{#2}}

\NewDocumentCommand{\depth}{s m}{\operatorname{depth}\IfBooleanTF{#1}{(#2)}{#2}}

\NewDocumentCommand{\Ho}{s m}{\operatorname{Ho}\left(#2\right)}
\NewDocumentCommand{\cone}{s m}{\operatorname{cone}\left(#2\right)}

\newcommand{\Hmo}{\operatorname{Hmo}}

\NewDocumentCommand{\id}{s o}{\mathbf{1}\IfValueT{#2}{_{#2}}}
\NewDocumentCommand{\Ext}{s o m m
  o}{\operatorname{Ext}\IfValueT{#2}{_{#2}}\IfValueT{#5}{^{#5}}\!\left(#3,#4\right)}
\NewDocumentCommand{\sExt}{s o m m o}{\underline{\operatorname{Ext}}\IfValueT{#2}{_{#2}}\IfValueT{#5}{^{#5}}\left(#3,#4\right)}
\NewDocumentCommand{\Hom}{s o m m o}{\operatorname{Hom}\IfValueT{#2}{_{#2}}\IfValueT{#5}{^{#5}}\!\left(#3,#4\right)}
\NewDocumentCommand{\Map}{s o m m o}{\operatorname{Map}\IfValueT{#2}{_{#2}}\IfValueT{#5}{^{#5}}\!\left(#3,#4\right)}
\NewDocumentCommand{\sHom}{s o m m o}{\underline{\operatorname{Hom}}\IfValueT{#2}{_{#2}}\IfValueT{#5}{^{#5}}\!\left(#3,#4\right)}
\NewDocumentCommand{\dgHom}{s o m m o}{\operatorname{hom}\IfValueT{#2}{_{#2}}\IfValueT{#5}{^{#5}}\!\left(#3,#4\right)}
\NewDocumentCommand{\End}{s o m
  o}{\operatorname{End}\IfValueT{#2}{_{#2}}\IfValueT{#4}{^{#4}}\!\left(#3\right)}
\NewDocumentCommand{\sEnd}{s o m o}{\underline{\operatorname{End}}\IfValueT{#2}{_{#2}}\IfValueT{#4}{^{#4}}\!\left(#3\right)}
\NewDocumentCommand{\RHom}{s o m m
  o}{\operatorname{RHom}\IfValueT{#2}{_{#2}}\IfValueT{#5}{^{#5}}\!\left(#3,#4\right)}
\NewDocumentCommand{\REnd}{s o m o}{\mathbf{R}\!\operatorname{End}\IfValueT{#2}{_{#2}}\IfValueT{#4}{^{#4}}\!\left(#3\right)}
\NewDocumentCommand{\dgFun}{s o m m o}{\operatorname{Fun}_{\mathrm{dg}}\IfValueT{#5}{^{#5}}\!\left(#3,#4\right)}

\NewDocumentCommand{\Sq}{o}{\operatorname{Sq}\IfValueT{#1}{\!\left(#1\right)}}

\NewDocumentCommand{\BK}{O{r} O{\bullet} O{*}}{E_{#1}^{#2,#3}}

\NewDocumentCommand{\HC}{s O{\bullet} O{*} m o}{\operatorname{C}^{#2\IfBooleanF{#1}{,#3}}\!\left(#4\IfValueT{#5}{,#5}\right)}
\NewDocumentCommand{\NHC}{s O{\bullet} O{*} m o}{\bar{\operatorname{C}}^{#2\IfBooleanF{#1}{,#3}}\!\left(#4\IfValueT{#5}{,#5}\right)}
\NewDocumentCommand{\HH}{s O{\bullet} O{*} m o}{\operatorname{HH}^{#2\IfBooleanF{#1}{,#3}}\!\left(#4\IfValueT{#5}{,#5}\right)}
\NewDocumentCommand{\SHC}{s O{\bullet} O{*} m o}{\operatorname{C}^{#2\IfBooleanF{#1}{,#3}}\!\left(#4\IfValueT{#5}{,#5}\right)}
\NewDocumentCommand{\SHH}{s O{\bullet} O{*} m o}{\operatorname{HH}^{#2\IfBooleanF{#1}{,#3}}\!\left(#4\IfValueT{#5}{,#5}\right)}
\NewDocumentCommand{\TateHH}{s O{\bullet} O{*} m o}{\underline{\operatorname{HH}}^{#2\IfBooleanF{#1}{,#3}}\!\left(#4\IfValueT{#5}{,#5}\right)}
\NewDocumentCommand{\TateExt}{s o m m o}{\underline{\operatorname{Ext}}\IfValueT{#2}{_{#2}}\IfValueT{#5}{^{#5}}\left(#3,#4\right)}

\NewDocumentCommand{\class}{s m}{\set{#2}}

\NewDocumentCommand{\simto}{s}{\IfBooleanTF{#1}{\stackrel{\sim}{\longrightarrow}}{\stackrel{\sim}{\to}}}

\NewDocumentCommand{\gLambda}{O{}}{\Lambda(\sigma,d)^{#1}}

\NewDocumentCommand{\AS}{s O{d} m m}{\operatorname{S}\IfBooleanT{#1}{^{\simeq}}_{#2}(#3,#4)}

\NewDocumentCommand{\opA}{O{\infty}}{\mathtt{A}_{#1}}

\newcommand{\abs}[1]{|#1|} % degree
\NewDocumentCommand{\uHH}{O{\bullet} m o}{\operatorname{HH}^{#1}\!\left(#2\IfValueT{#3}{,#3}\right)} % ungraded Hochschild cohomology
\NewDocumentCommand{\uTateHH}{O{\bullet} m o}{\underline{\operatorname{HH}}^{#1}\!\left(#2\IfValueT{#3}{,#3}\right)} % ungraded Hochschild--Tate cohomology
\NewDocumentCommand{\uHC}{O{\bullet} m o}{\operatorname{C}^{#1}\!\left(#2\IfValueT{#3}{,#3}\right)} % ungraded Hochschild complex
\NewDocumentCommand{\RHC}{s O{\bullet} O{*} m o}{\operatorname{C}_{\CC[\imath^{\pm1}]}^{#2\IfBooleanF{#1}{,#3}}\!\left(#4\IfValueT{#5}{,#5}\right)} % relative Hochschild complex
\DeclareRobustCommand\alto{n} % equation superscript
\DeclareRobustCommand\bajo{} % equation subscript
\usepackage{tikz} % circled number
\newcommand*\circled[1]{\tikz[baseline=(char.base)]{
            \node[shape=circle,draw,inner sep=1pt] (char) {\scriptsize #1};}}
\newcounter{term}
\AtBeginDocument{%
  \let\mylabel\label
}
\newcommand{\mytag}[1]{%
  \begingroup % keep the effects of \refstepcounter local
    \refstepcounter{term}%
    \mylabel{#1}%
    \text{\circled{\theterm}}%
  \endgroup
}
\newcommand{\refcirc}[1]{\circled{\ref{#1}}}

\newcommand{\EulerDer}[1]{\bar{\delta}}
\newcommand{\EulerClass}[1]{\delta}

\usepackage[nameinlink]{cleveref}%

\usepackage{thmtools}%

\numberwithin{equation}{subsection}

\newcounter{thmintro}%

\declaretheorem[sibling=thmintro,name=Theorem]{theorem intro}%
\declaretheorem[sibling=equation]{theorem}%
\declaretheorem[numbered=no,name=Theorem]{theorem*}%
\declaretheorem[sibling=equation]{proposition}%
\declaretheorem[sibling=equation]{lemma}%
\declaretheorem[sibling=equation]{corollary}%
\declaretheorem[style=definition,sibling=equation]{definition}%
\declaretheorem[style=definition,sibling=thmintro]{conjecture}%
\declaretheorem[style=remark,sibling=equation]{setting}%
\declaretheorem[style=remark,sibling=equation]{remark}%
\declaretheorem[style=remark,numbered=no,name=Remark]{remark*}%
\declaretheorem[style=remark,numbered=no,name=Example]{example*}%

\usetikzlibrary{matrix,arrows,calc}

\newcounter{rpage} 
\newcounter{trunco}
\DeclareSymbolFont{sfoperators}{OT1}{cmss}{m}{n}%
\DeclareSymbolFontAlphabet{\mathsf}{sfoperators}%
\makeatletter%
\def\operator@font{\mathgroup\symsfoperators}%
\makeatother

\let\oldpentagon\pentagon%
\renewcommand{\pentagon}{{\normalfont\oldpentagon}}

\title[The Donovan--Wemyss Conjecture]{The Donovan--Wemyss Conjecture via the
  Derived Auslander--Iyama Correspondence}



\author[G.~Jasso]{Gustavo Jasso}%
\address[G.~Jasso]{%
  Lund University, %
  Centre for Mathematical Sciences, %
  Sölvegatan 18A, %
  22100 Lund, %
  Sweden%
}%
\email{gustavo.jasso@math.lu.se}%
\urladdr{https://gustavo.jasso.info}

\author[B.~Keller]{Bernhard Keller}%
\address[B.~Keller]{%
Université Paris Cité and Sorbonne Université, CNRS, IMJ-PRG, F-75013, France }%
\email{bernhard.keller@imj-prg.fr }%
\urladdr{https://webusers.imj-prg.fr/~bernhard.keller/}

\author[F.~Muro]{Fernando Muro}%
\address[F.~Muro]{%
  Universidad de Sevilla, %
  Facultad de Matemáticas, %
  Departamento de Álgebra, %
  Calle Tarfia s/n, %
  41012 Sevilla, %
  Spain%
}%
\email{fmuro@us.es}%
\urladdr{https://personal.us.es/fmuro/}

\NewDocumentCommand{\Con}{s}{\Lambda\IfBooleanT{#1}{_{\mathrm{con}}}}
\NewDocumentCommand{\DerCon}{s}{\mathbf{\Lambda}\IfBooleanT{#1}{_{\mathrm{con}}}}

\newcommand{\OO}{\mathcal{O}}

\newcommand{\NN}{\mathcal{N}}

\begin{document}

\sloppy

\begin{abstract}
  We provide an outline of the proof of the Donovan--Wemyss Conjecture in the
  context of the Homological Minimal Model Program for threefolds. The proof
  relies on results of August, of Hua and the second-named author, Wemyss, and
  on the Derived Auslander--Iyama Correspondence---a recent result by the first-
  and third-named authors.
\end{abstract}

\maketitle

\setcounter{tocdepth}{1}
\tableofcontents

\crefname{question}{Question}{Questions}
\Crefname{question}{Question}{Questions}
\crefname{conjecture}{Conjecture}{Conjectures}
\Crefname{conjecture}{Conjecture}{Conjectures}

\section*{Introduction}

We work over the field $\CC$ of complex numbers. A \emph{compound Du Val (=cDV)
  singularity} is a complete local hypersurface
\[
  R\cong\frac{\CC\llbracket x,y,z,t\rrbracket}{(f+tg)},
\]
where $\CC\llbracket x,y,z\rrbracket/(f)$ is a Kleinian surface singularity and
$g\in\CC\llbracket x,y,z,t\rrbracket$ is arbitrary. Introduced by Reid in the
early 1980s~\cite{Rei83}, cDV singularities form an important class of
three-dimensional singularities in birational geometry and play a significant
role in the Minimal Model Program (MMP) for threefolds~\cite[Sec.~5.3]{KM98} as
well as in the Homological MMP~\cite{Wem18}. We refer the reader
to~\cite[Ch.~1]{August19} and~\cite{Wem23} for introductions to the subject.

This note is concerned with the following geometric situation: Let $R$ be an
isolated cDV singularity and $p\colon X\to\Spec*{R}$ a crepant resolution, that
is $p$ is a proper birational map with smooth source such that the pullback of
the dualising sheaf of $\Spec*{R}$ along $p$ is the dualising sheaf of $X$. It
follows that $\Spec*{R}$ has a unique singular point $\mathfrak{m}$ and the
(reduced) exceptional fibre $p^{-1}(\mathfrak{m})=\bigcup_{i=1}^{n}C_i$ is a
union of curves, with $C_i\cong\mathbb{P}_\CC^1$~\cite[Lemma~3.4.1]{VdBer04}. To
these data, Donovan and Wemyss~\cite{DW16,DW19a} associate a (basic, connected)
finite-dimensional algebra $\Con*=\Con*(p)$, the \emph{contraction algebra of
  $p$}, which represents the functor of `simultaneous non-commutative
deformations' of the reduced exceptional fibre. By construction, $\Con*$ is a
$\CC^n$-augmented algebra, and hence in particular determines the number $n$ of
irreducible components of the exceptional fibre. The contraction algebra encodes
a surprising amount of information stemming from the given geometric setup. For
example, when $p$ contracts a single curve, the contraction algebra recovers
known invariants such as Reid's width~\cite{Rei83} and the Gopakumar--Vafa
invariants~\cite{Kat08}, see~\cite{Tod15}. Neither the dimension nor the Gabriel
quiver of contraction algebras suffice for differentiating cDV
singularities~\cite[Table~2]{DW16}. In fact, it is well-known that there are
continuous families of pairwise non isomorphic cDV singularities (that is `cDV
singularities have moduli'). Notwithstanding, at the risk of stating the
obvious, let us point out that the contraction algebra is equiped with crucial
data in the form of the multiplication law and that this law is essential in
recovering the above mentioned invariants. Equipped with their algebra
structure, contraction algebras distinguish between non-isomorphic isolated cDV
singularities that admit a crepant resolution in all known examples. These
considerations motivate the following remarkable conjecture.

\begin{conjecture}[Donovan and Wemyss {\cite{DW16}}]
  \label{conj:DW}
  Let $R_1$ and $R_2$ be isolated cDV singularities with crepant resolutions
  \[
    p_1\colon X_1\to \Spec*{R_1}\qquad\text{and}\qquad p_2\colon X_2\to
    \Spec*{R_2}.
  \]
  Then, the contraction algebras $\Con*(p_1)$ and $\Con*(p_2)$ are derived
  equivalent if and only if there is an isomorphism of algebras $R_1\cong R_2$.
\end{conjecture}

The original conjecture was formulated only in the case of single-curve
contractions; algebraically, this corresponds to the case where the contraction
algebras are local and thus derived equivalence reduces to mere isomorphism of
algebras since contraction algebras are basic, see~\cite[Prop.~6.7.4]{Zim14} for
example. In the above form, which allows for contracting multiple curves, the
conjecture appeared in print in~\cite[Conj.~1.3]{Aug20a}.

That the contraction algebras of a given isolated cDV singularity are derived
equivalent follows by combining results from Wemyss~\cite{Wem18} and
Dugas~\cite{Dugas15}. In this note we provide an outline of the proof of the
remaining part of \Cref{conj:DW}. This proof first appeared in the appendix
to~\cite{JKM22} written by the second-named author where it is explained how the
conjecture follows by combining previous results of August~\cite{Aug20a}
and~\cite{HK18} with the Derived Auslander--Iyama Co\-rres\-pon\-dence---the
main result in~\cite{JKM22}. For the sake of concreteness, we restrict ourselves
to the specific context of the conjecture, with the understanding that most
concepts and results that are presented here are but special cases of a much
more general theory that the reader can find in the original sources. We hope
that this sacrifice in generality makes the proof of the conjecture accessible
to a broader readership.

The proof of the conjecture makes use of an invariant associated to a
contraction algebra of a resolution of a cDV singularity $R$ that we call the
restricted universal Massey product. This is a certain Hochschild cohomology
class that is induced by the first possibly non-trivial higher operation on a
minimal $A_\infty$-algebra model of the derived endomorphism algebra of a
generator of the singularity category of $R$. As it turns out, this invariant
determines the derived endomorphism algebra of the generator up to
quasi-isomorphism and, combined with the aforementioned results, this is the
final technical ingredient in the proof of~\Cref{conj:DW}.

\subsection*{Acknowledgements}

The authors thank the organisers of the Online Workshop on Higher Dimensional
Homological Algebra at the IPM-Isfahan in May 2022, where the application of the
Derived Auslander--Iyama Correspondence to the Donovan--Wemyss Conjecture was
first found. The authors thank Michael Wemyss and Zheng Hua for interesting
conversations and email exchanges; in addition, they thank M.~W. for suggesting
the terminology `$2\ZZ$-derived contraction algebra' used in these proceedings.
The third-named author is also grateful to Matt Booth for answering questions.

\subsection*{Financial support}

F.~M.~was partially supported by grants PID2020-117971GB-C21
funded by MCIN/AEI/10.13039/501100011033 and
P20\_01109 (JUNTA/FEDER, UE).

\section{Preliminaries}

In this section we collect preliminary definitions and results that are needed in
our proof of~\Cref{conj:DW}. We use freely the theories of differential graded
categories~\cite{Kel94,Kel06} and $A_\infty$-categories~\cite{Lef03}. We denote
the derived category of an algebra or, more generally, a DG algebra $A$ by
$\DerCat{A}$; the perfect derived category of $A$, that is the full subcategory
of $\DerCat{A}$ spanned by its compact objects, is denoted by $\DerCat[c]{A}$.
All (DG) modules are right (DG) modules.

\subsection{$2\ZZ$-cluster tilting objects}

Let $\T$ be a triangulated category whose underlying additive category is
Krull--Schmidt and has finite-dimensional morphism spaces.

\begin{definition}[{\cite{IY08,GKO13}}] A basic\footnote{An object in a
    Krull--Schmidt additive category is basic if all of its indecomposable
    direct summands have multiplicity one.} object $T\in\T$ is \emph{$2$-cluster
    tilting} if the following conditions hold:
  \begin{enumerate}
    \item The object $T$ is \emph{rigid}: $\T(T,T[1])=0$.
    \item For each object $X\in\T$ there exists a triangle $T_1\to T_0\to X\to
            T_1[1]$ with $T_0,T_1\in\add*{T}$, where $\add*{T}$ is the smallest additive
          subcategory of $\T$ containing $T$ that is closed under direct summands.
  \end{enumerate}
  We say that $T$ is \emph{$2\ZZ$-cluster tilting} if it is $2$-cluster tilting
  and $T\cong T[2]$.
\end{definition}

\begin{remark}
  Clearly, if $T\in\T$ is a $2$- or $2\ZZ$-cluster tilting object, then $T$
  generates $\T$ as a triangulated category with split idempotents (which is to
  say that $T$ is a classical generator of $\T$). In particular, if the
  triangulated category $\T$ is algebraic\footnote{A triangulated category is
    \emph{algebraic} if it is equivalent---as a triangulated category---to the
    stable category of a Frobenius exact category~\cite{Kel94}.} then there
  exists a DG algebra $A$ and an equivalence of triangulated categories
  \[
    \T\stackrel{\sim}{\longrightarrow}\DerCat[c]{A},\qquad T\longmapsto A.
  \]
\end{remark}

\begin{remark}
  Given a basic $2$-cluster tilting object $T\in\T$, one can produce a new such
  object by a procedure called \emph{mutation} that, in a nutshell, replaces a
  single indecomposable direct summand of $T$ by a new one, see~\cite{IY08} for
  a precise definition. This process, which can be iterated, is an important
  reason for introducing $2$-cluster tilting objects into the framework of the
  Homological MMP, see~\cite{Wem18} and compare
  with~\Cref{sec:contraction_algebras}.
\end{remark}

\begin{remark}
  In general, $2\ZZ$-cluster tilting objects are \emph{not} invariant under
  mutation (however, see~\cite[Sec.~4]{HI11}). In the context of the Homological
  MMP this problem does not occur since the notions of $2$- and $2\ZZ$-cluster
  tilting object coincide, see~\Cref{sec:MCM_sing}.
\end{remark}

\subsection{Maximal Cohen--Macaulay modules and singularity categories}
\label{sec:MCM_sing}

Let $R$ be an isolated cDV singularity and
\[
  \CM*{R}\coloneqq\set{M\in\mmod*{R}}[\depth*{M}=\dim(R)]
\]
be the category of \emph{maximal Cohen--Macaulay $R$-modules} \cite{Yos90,LW12}.
The category $\CM(R)$ is a Frobenius exact category and, therefore, the stable
category $\stableCM*{R}$ has an induced structure of a triangulated category;
moreover, there is a canonical equivalence of triangulated categories
\[
  \stableCM*{R}\stackrel{\sim}{\longrightarrow}\SingCat{R}\coloneqq\DerCat[b]{\mmod*{R}}/\KCh[b]{\proj*{R}},
\]
where $\SingCat{R}$ is the \emph{singularity category of $R$}, see~\cite{Buc21}
for details. We record the following facts for later use:
\begin{itemize}
  \item \cite[Prop.~1.18]{Yos90} Since $R$ is complete local, $\stableCM*{R}\simeq\SingCat{R}$ is a Krull--Schmidt category
        with finite-dimensional morphism spaces.
  \item \cite{Aus78} Since $R$ is $3$-dimensional,
        $\stableCM*{R}\simeq\SingCat{R}$ is a
        \emph{$2$-Calabi--Yau} triangulated category~\cite{Kon98,Kel08}, that is there is a natural
        isomorphism
        \[
          D\!\Hom{X}{Y}\stackrel{\simeq}{\longrightarrow}\Hom{Y}{X[2]},\qquad
          X,Y\in\stableCM*{R}\simeq\SingCat{R},
        \]
        where $V\mapsto DV$ denotes the passage to the $\CC$-linear dual.
  \item \cite{Eis80} Since $R$ is a hypersurface, $\stableCM*{R}\simeq\SingCat{R}$
        is a $2$-periodic triangulated category, that is there is an isomorphism of exact functors
        $[2]\cong\id$. In particular, the notions of $2$- and $2\ZZ$-cluster tilting
        object coincide in this context.
  \item The endomorphism algebra of any basic object $X$ in $\stableCM*{R}\simeq\SingCat{R}$ is a
        symmetric algebra. This is an immediate consequence of
        the natural isomorphisms
        \[
          \Hom{X}{X}\cong D\Hom{X}{X[2]}\cong D\Hom{X}{X}.
        \]
\end{itemize}
We also consider the DG category $\dgSingCat{R}$, which is defined as the DG
quotient~\cite{Kel99,Dri04} of the canonical DG enhancements of the triangulated
categories $\DerCat[b]{\mmod*{R}}$ and $\KCh[b]{\proj*{R}}$. By construction,
\[
  \dgH[0]{\dgSingCat{R}}=\SingCat{R}.
\]

\subsection{Contraction algebras via $2\ZZ$-cluster tilting objects}
\label{sec:contraction_algebras}

Let $R$ be an isolated cDV singularity and ${p\colon X\to \Spec*{R}}$ a crepant
resolution. As explained in the introduction, to this geometric setup Donovan
and Wemyss associate a basic finite-dimensional algebra $\Con*=\Con*(p)$. We
recall an alternative construction of the algebra $\Con*$ that is more adapted
to the methods we utilise in this note. Given $p$ as above, a theorem of Van den
Bergh~\cite{VdBer04} furnishes a tilting bundle
\[
  \OO_X\oplus \NN=\OO_X\oplus\NN(p)\in\coh*{X}
\]
and Wemyss proves~\cite{Wem18} that there is an isomorphism of algebras
\[
  \Con*\cong\sEnd[R]{N}
\]
between the contraction algebra of $p$ and the stable endomorphism algebra of
${N\coloneqq H^0(\NN)\in\CM*{R}}$. Remarkably, when viewed as an object of the
triangulated category $\stableCM*{R}$, the $R$-module $N$ is a $2\ZZ$-cluster
tilting object. Conversely, given a $2\ZZ$-cluster tilting object
$T\in\stableCM*{R}$, there exists a crepant resolution of $\Spec*{R}$ whose
associated contraction algebra is isomorphic to $\sEnd[R]{T}$. We summarise the
previous discussion in the following theorem.\footnote{Wemyss proves even more:
  Up to isomorphism on both sides, crepant resolutions of $R$ correspond
  bijectively to (basic) $2\ZZ$-cluster tilting objects in $\SingCat{R}$. In
  particular, the number of isomorphism classes of $2\ZZ$-cluster tilting
  objects in $\SingCat{R}$ if finite, for the number of minimal models of
  $\Spec*{R}$ is finite~\cite{KM87}.}

\begin{theorem}[{\cite{Wem18}}]
  \label{thm:Wemyss}
  Let $R$ be an isolated cDV singularity and assume that $\Spec*{R}$ admits a
  crepant resolution. Then, the contraction algebras of $R$ are precisely the
  endomorphism algebras of $2\ZZ$-cluster tilting objects in the triangulated
  category $\stableCM*{R}\simeq\SingCat{R}$.
\end{theorem}

The following theorem of August reduces \Cref{conj:DW} from a derived
equivalence to an isomorphism problem. The proof leverages the characterisation
of contraction algebras provided by \Cref{thm:Wemyss}.

\begin{theorem}[{\cite[Thm.~1.4]{Aug20a}}]
  \label{thm:August}
  Let $R$ be an isolated cDV singularity and assume that $\Spec*{R}$ admits a
  crepant resolution. The contraction algebras of $R$ form a single and complete
  derived equivalence class of basic algebras.
\end{theorem}

\begin{corollary}
  \label{cor:August}
  Let $R_1$ and $R_2$ be isolated cDV singularities with crepant resolutions
  \[
    p_1\colon X_1\to \Spec*{R_1}\qquad\text{and}\qquad p_2\colon X_2\to
    \Spec*{R_2}
  \]
  and corresponding contraction algebras
  $\Con_1=\Con*(p_1)$ and $\Con_2=\Con*(p_2)$. If the algebras $\Lambda_1$ and
  $\Lambda_2$ are derived equivalent, then there exists a contraction algebra
  $\Lambda$ of $R_2$ such that $\Lambda\cong\Lambda_1$.
\end{corollary}

\subsection{Hochschild cohomology}\label{subsec:HH}

Let $A$ be a graded algebra. The bigraded \emph{Hochschild cochain complex} has components
\[
  \HC[p][q]{A}[A]\coloneqq\Hom[\CC]{A^{\otimes
        p}}{A(q)},\qquad p\geq 0,\ q\in\ZZ,
\]
where $V\mapsto V(1)$ is the (vertical) degree shift of graded vector spaces,
equipped with the Hochschild differential $x\mapsto \partial(x)$ of bidegree
$(1,0)$, see~\cite{Mur20b} for the precise definition, related structure
(described below) and sign conventions. The first degree is called
\emph{horizontal} or \emph{Hochschild degree} and the second is the
\emph{vertical} or \emph{internal degree}; the sum of both is the \emph{total
  degree} and we denote it by $\abs{x}$ (we also use this notation for the
degree of an element in a singly-graded vector space). The component of total degree $n$
of the Hochschild complex is
\[
  \prod_{p+q=n}\HC[p][q]{A}[A].
\]

The Hochschild complex is equipped with a \emph{brace algebra} structure,
consisting of operations called \emph{braces} (which we do not describe
explicitly here)
\begin{align*}
  \HC{A}[A]^{\otimes^{n+1}}              & \longrightarrow \HC{A}[A],         \\
  x_0\otimes x_1\otimes\cdots\otimes x_n & \longmapsto x_0	\{x_1,\dots,x_n\},
\end{align*}
that are defined for all $n\geq 1$, and satisfy the \emph{brace relation}
\begin{multline*}
  x\{y_1,\dots,y_p\}\{z_1,\dots,z_q\}\\
  =\sum_{0\leq i_1\leq j_1\leq\cdots\leq i_p\leq j_p\leq q}(-1)^\epsilon x\{z_1,\dots,z_{i_1},y_1\{z_{i_1+1},\dots,z_{j_1}\},z_{j_1+1},\dots\\
  \dots, z_{i_p},y_q\{z_{i_p+1},\dots,z_{j_p}\},z_{j_p+1},\dots,z_q\}.
\end{multline*}
Above, $\epsilon$ reflects the Koszul sign rule with respect to the total degree
shifted by $-1$. Brace operations have horizontal degree $-n$ and vertical
degree $0$ and the $n$-th brace operation vanishes when $x_0$ has
horizontal degree $<n$.
The Hochschild complex is a bigraded
associative algebra equipped with the \emph{cup-product}
\[x\cdot y=(-1)^{\abs{x}-1}m_2\{x,y\},\]
where $m_2\in\HC[2][0]{A}[A]$ is, up to a sign, the multiplication in the graded algebra $A$.
The Hochschild complex also has the structure of a (horizontally-shifted)
graded Lie algebra, with Lie bracket
\[[x,y]=x\{y\}-(-1)^{(\abs{x}-1)(\abs{y}-1)}y\{x\}\] of horizontal degree $-1$
and vertical degree $0$. The Hochschild differential is given by
\[\partial(x)=[m_2,x]\]
and satisfies the corresponding Leibniz rules with respect to the previous
associative product and Lie bracket.

The relation between brace operations, the Hochschild differential and the cup
product are encoded in the following straightforward consequences of the brace
relation.

\begin{lemma}\label{braces_and_differential}
  The following formula holds for all $n\geq 1$:
  \begin{align*}
    \partial(x_0\{x_1,\dots,x_n\}) & =\partial(x_0) \{x_1,\dots,x_n\}                                                                          \\
                                   & \phantom{=}+\sum_{i=1}^n (-1)^{\sum_{j=0}^{i-1}\abs{x_j}-i}x_0\{x_1,\dots,\partial(x_i),\dots,x_n\}       \\
                                   & \phantom{=}+(-1)^{\abs{x_0}-1+\abs{x_0}\abs{x_1}}x_1\cdot x_0\{x_2,\dots,x_{n-1}\}                        \\
                                   & \phantom{=}+\sum_{i=1}^{n-1}(-1)^{\sum_{j=0}^i\abs{x_i}-i-1} x_0\{x_1,\dots, x_i\cdot x_{i+1},\dots,x_n\} \\
                                   & \phantom{=}+(-1)^{\sum_{j=0}^{n-1}\abs{x_j}-n-1}x_0\{x_1,\dots,x_{n-1}\}\cdot x_n.
  \end{align*}
\end{lemma}

\begin{lemma}\label{braces_and_cup_product}
  The following formula holds for all $n\geq 1$:
  \begin{align*}
    (x\cdot y)\{z_1,\dots,z_n\}=\sum_{i=0}^{n}(-1)^{\abs{y}\sum_{j=1}^i(\abs{z_i}-1)} x\{z_1,\dots,z_i\}\cdot y\{z_{i+1},\dots,z_n\}.
  \end{align*}
\end{lemma}

\Cref{braces_and_differential} for $n=1$ proves that the induced associative
product in \emph{Hochschild cohomology}
\[
  \HH{A}[A]=\dgH[\bullet,*]{\HC{A}[A]}
\]
(the cohomology of the
Hochschild complex) is graded commutative with respect to the total degree. This
products satisfies the following compatibility relation with the (horizontally
shifted) Lie algebra structure,
\[ [x,y\cdot z]=[x,y]\cdot z+(-1)^{(\abs{x}-1)\abs{y}}y\cdot[x,z],
\]
and hence Hochschild cohomology is a \emph{Gernstenhaber algebra}. For this we use both \Cref{braces_and_differential,braces_and_cup_product}, for $n=2$ and $n=1$ respectively. The
Hochschild complex $\HC{A}[M]$ and Hochschild cohomology $\HH{A}[M]$ are defined, more
generally, for $M$ an $A$-bimodule, but it does not have any
multiplicative structure in this general case. For the existence of a graded
associative algebra structure it suffices that $M$ be an associative algebra in
$A$-bimodules, see also \cite[Sec.~1]{Mur22}.

\subsection{Minimal $A_\infty$-algebras}

We now describe minimal $A_\infty$-algebras and their morphisms in terms of
the Hochschild complex. A \emph{minimal $A_\infty$-algebra structure} on a
graded algebra $A$ is a Hochschild cochain
\[m=(0,0,0,m_3,\dots,m_n,\dots)\in \HC{A}[A]\] of total degree $2$ such that the
Maurer--Cartan equation
\begin{equation}
  \label{eq:MC}
  \partial(m)+m\{m\}= \partial(m)+\tfrac{1}{2}[m,m]=0
\end{equation}
is satisfied. The pair $(A,m)$ is also denoted
\[(A,m_3,\dots,m_n,\dots).\] If $g\colon A'\to A$ is a graded algebra
isomorphism, then
\[m*g= (0,0,0,g^{-1}m_3g^{\otimes^3},\dots, g^{-1} m_n g^{\otimes^n},\dots)\] is
a minimal $A_\infty$-algebra structure on $A'$. If
\[m'=(0,0,0,m'_3,\dots,m_n',\dots)\in \HC{A}[A]\] is another minimal
$A_\infty$-algebra structure on $A$, an \emph{$A_\infty$-isomorphism with
  identity linear part}
\[f\colon (A,m)\longrightarrow (A,m')\] is a Hochschild cochain
\[f=(0,0,f_2,f_3,\dots,f_n,\dots)\in \HC{A}[A]\] of total degree $1$ such that
the following Hochschild cochain vanishes
\begin{equation}\label{A-map}
  \partial(f)+f\cdot f+\sum_{r\geq 0}m'\{f,\stackrel{r}{\dots},f\}-	m-f\{m\}.
\end{equation}
More generally, an \emph{$A_\infty$-isomorphism} between minimal
$A_\infty$-algebras
\[f\colon (A,m)\longrightarrow (A',m')\] consists of an isomorphism of graded
algebras
\[f_1\colon A\longrightarrow A'\] and a Hochschild cochain
\[(0,0,f_2,f_3,\dots,f_n,\dots)\in \HC{A}[A']\] of total degree $1$ such that
\[(0,0,f_1^{-1}f_2, f_1^{-1} f_3,\dots, f_1^{-1} f_n,\dots)\colon
  (A,m)\longrightarrow (A,m'*f_1)\] is an $A_\infty$-isomorphism with
identity linear part.

\section{The Derived Donovan--Wemyss Conjecture}

In this section we discuss one of the main results in~\cite{HK18}---crucial to
our proof of the Dononvan--Wemyss conjecture---and explain how it implies a
\emph{derived version} of the conjecture (\Cref{cor:HK}).

\subsection{$2\ZZ$-derived contraction algebras}

By means of the equivalence of triangulated categories
$\stableCM*{R}\simeq\SingCat{R}$, the contraction algebra associated to a
crepant resolution $p$ of an isolated cDV singularity can be promoted to the DG
algebra
\[
  \DerCon*=\DerCon*(p)\coloneqq\REnd{N}
\]
given by the derived endomorphism algebra of the corresponding $2\ZZ$-cluster
tilting object $N=N(p)\in\dgSingCat{R}$. By construction
$\dgH[0]{\DerCon*}\cong\Con*$ and, as a consequence of the $2$-periodicity of
the singularity category of $R$,
\[
  \dgH{\DerCon*}\cong\Con*[\imath^{\pm1}]=\Con*\otimes_\CC\CC[\imath^{\pm1}],\qquad
  |\imath|=-2.
\]
We refer to the DG algebra $\DerCon*$ as the \emph{$2\ZZ$-derived contraction algebra
  of $p$}. The (soft) non-positive truncation
$\DerCon*^{\leq0}=\tau^{\leq0}\DerCon*$ of $\DerCon*$ is quasi-isomorphic to the
\emph{derived contraction algebra of $p$} considered for example in~\cite{Boo19,Boo21,HK18}, and there is an isomorphism of graded algebras
\[
  \dgH{\DerCon*^{\leq0}}\cong\Con*[\imath]=\Con*\otimes_\CC\CC[\imath],\qquad
  |\imath|=-2.
\]
The $2\ZZ$-derived contraction algebra $\DerCon*$ is a
localisation of $\DerCon*^{\leq0}$ (see~\cite[Thms.~6.4.6 and 7.2.3]{Boo21} and
\cite[Thm.~4.17]{HK18}) and $\DerCon*$ can also be interpreted as a
\emph{non-connective} variant of $\DerCon*^{\leq0}$.

Notice also that, since $2\ZZ$-cluster tilting objects are in particular
classical generators, there is a canonical quasi-equivalence of DG categories
\[
  \DerCat[c]{\DerCon*}_{\mathrm{dg}}\stackrel{\sim}{\longrightarrow}\dgSingCat{R}
\]
that induces an equivalence of triangulated categories
\[
  \DerCat[c]{\DerCon*}\stackrel{\sim}{\longrightarrow}\SingCat{R}.
\]
Although we do not need this fact in the sequel, we mention that there is an
equivalence of triangulated categories\footnote{For a DG algebra $A$, we denote
  by $\DerCat[fd]{A}$ the full subcategory of $\DerCat{A}$ spanned by the DG
  $A$-modules with finite-dimensional total cohomology.}~\cite[Theorem~4.17 and Lemma~5.12]{HK18}
\[
  \C(\DerCon*^{\leq 0})\coloneqq\DerCat[c]{\DerCon*^{\leq
      0}}/\DerCat[fd]{\DerCon*^{\leq 0}}\simeq\SingCat{R},
\]
that is compatible with the canonical DG enhancements on either side. The
category $\C(\DerCon*^{\leq 0})$ is known as the \emph{Amiot cluster category of
  $\DerCon*^{\leq 0}$}~\cite{Ami07} and, indeed, establishing a link between
birational geometry and the theory of cluster categories was one of the
objectives in~\cite{HK18}.

\subsection{The Derived Donovan--Wemyss Conjecture}

The following theorem of Hua and the second-named author settles a derived
version of \Cref{conj:DW}.

\begin{theorem}[{\cite[Thm.~5.9]{HK18}}]
  \label{thm:HK}
  Let $R=\CC\llbracket x,y,z,t\rrbracket/(f)$ be an isolated cDV singularity.
  Then, there is an isomorphism of algebras
  \[
    \HH*[0]{\dgSingCat{R}}\cong\frac{\CC\llbracket
      x,y,z,t\rrbracket}{\left(f,\partial_xf,\partial_y f,\partial_z
      f,\partial_tf\right)}
  \]
  between the $0$-th Hochschild cohomology of the DG category $\dgSingCat{R}$
  and the Tyurina algebra of $R$. In particular, if $R'$ is a further isolated
  cDV singularity such that the DG categories $\dgSingCat{R}$ and
  $\dgSingCat{R'}$ are quasi-equivalent, then there is an isomorphism of
  algebras $R\cong R'$.
\end{theorem}

\begin{remark}
  The proof of \Cref{thm:HK} relies on a comparison result \cite{Kel18b,Kel19}
  between the singular Hochschild cohomology (=Hochschild--Tate cohomology) of
  $R$ and the Hochschild cohomology of the DG category $\dgSingCat{R}$. The
  appearance of the Tyurina algebra stems from an earlier result of the Buenos
  Aires Cyclic Homology Group~\cite[Thm.~3.2.7]{BCHG92}. That the Tyurina
  algebra, together with the dimension of $R$, determines the isomorphism type
  of the singularity is a theorem of Mather and Yau~\cite{MY82}.
\end{remark}

\begin{remark}
  In~\cite{Dyc11}, Dyckerhoff shows that the $0$-th Hochschild cohomology of
  $\dgSingCat{R}$---viewed as a $\ZZ/2$-graded DG category---is isomorphic to
  the Milnor algebra
  \[
    \frac{\CC\llbracket x,y,z,t\rrbracket}{\left(\partial_xf,\partial_y
      f,\partial_z f,\partial_tf\right)}
  \]
  of the singularity (which does not determine the isomorphism type of the
  singularity, even if one knows the dimension). Thus, in~\Cref{thm:HK} it is
  crucial to consider $\dgSingCat{R}$ as a $\ZZ$-graded DG category.
\end{remark}

\begin{corollary}[Derived Donovan--Wemyss Conjecture]
  \label{cor:HK}
  Let $R_1$ and $R_2$ be isolated cDV singularities with crepant resolutions
  \[
    p_1\colon X_1\to \Spec*{R_1}\qquad\text{and}\qquad p_2\colon X_2\to
    \Spec*{R_2}.
  \]
  If the $2\ZZ$-derived contraction algebras $\DerCon*(p_1)$ and $\DerCon*(p_2)$ are
  quasi-isomorphic, then there is an isomorphism of algebras $R_1\cong R_2$.
\end{corollary}
\begin{proof}
  Indeed, if the DG algebras $\DerCon*(p_1)$ and $\DerCon*(p_2)$ are
  quasi-isomorphic, then the DG categories
  \[
    \DerCat[c]{\DerCon*(p_1)}_{\mathrm{dg}}\simeq\dgSingCat{R_1}\qquad\text{and}\qquad
    \DerCat[c]{\DerCon*(p_2)}_{\mathrm{dg}}\simeq\dgSingCat{R_2}
  \]
  are quasi-equivalent. \Cref{thm:HK} then implies the existence of an
  isomorphism of algebras $R_1\cong R_2$.
\end{proof}

\section{Uniqueness of the $2\ZZ$-derived contraction algebra}

In this section we prove that $2\ZZ$-derived contraction algebras are determined up to
quasi-isomorphism by their zeroth cohomology plus a minimal amount of additional
algebraic data (see~\Cref{coro:rUMP_determines} for the precise statement).
Before that, we formulate a closely related result (\Cref{thm:JM-con}) that
states that two $2\ZZ$-derived contraction algebras whose zeroth cohomologies are
isomorphic as algebras must be quasi-isomorphic, and use this result to prove
\Cref{conj:DW}.

\subsection{Proof of the Donovan--Wemyss Conjecture}

In view of \Cref{cor:August,cor:HK}, \Cref{conj:DW} is an immediate consequence
of the following theorem, the proof of which is given in \Cref{sec:the_proof}.

\begin{theorem}
  \label{thm:JM-con}
  Let $R_1$ and $R_2$ be isolated cDV singularities with crepant resolutions
  \[
    p_1\colon X_1\to \Spec*{R_1}\qquad\text{and}\qquad p_2\colon X_2\to
    \Spec*{R_2}.
  \]
  If the contraction algebras $\Lambda(p_1)$ and $\Lambda(p_2)$ are isomorphic,
  then the $2\ZZ$-derived contraction algebras $\DerCon*(p_1)$ and $\DerCon*(p_2)$ are
  quasi-isomorphic.
\end{theorem}

\begin{remark}
  \Cref{thm:JM-con} is a special case of~\cite[Thm.~5.1.10]{JKM22},
  see~\Cref{sec:der_AIC}.
\end{remark}

\begin{proof}[Proof of \Cref{conj:DW} using \Cref{thm:JM-con}] Let $R_1$ and
  $R_2$ be isolated cDV singularities with crepant resolutions
  \[
    p_1\colon X_1\to \Spec*{R_1}\qquad\text{and}\qquad p_2\colon X_2\to
    \Spec*{R_2}
  \]
  whose corresponding contraction algebras $\Con*(p_1)$ and $\Con*(p_2)$ are
  derived equivalent. In view of~\Cref{cor:August,cor:HK}, we may and we will
  assume that $\Con*(p_1)$ and $\Con*(p_2)$ are isomorphic and hence,
  by~\Cref{thm:JM-con}, the $2\ZZ$-derived contraction algebras~$\DerCon*(p_1)$ and
  $\DerCon*(p_2)$ are quasi-isomorphic. Finally, \Cref{cor:HK} yields the
  desired algebra isomorphism $R_1\cong R_2$.
\end{proof}

\subsection{The restricted universal Massey product}

The proof of \Cref{thm:JM-con} makes use of an invariant of the $2\ZZ$-derived
contraction algebra, a certain Hochschild cohomology class of bidegree $(4,-2)$
that we call the restricted universal Massey product. As we explain below, this
invariant is induced by the first possibly non-trivial higher operation on a
minimal $A_\infty$-algebra model of the $2\ZZ$-derived contraction algebra.

\begin{setting}
  Fix an isolated cDV singularity $R$ that admits a crepant resolution ${p\colon
        X\to\Spec*{R}}$, and let $\DerCon=\DerCon*(p)$ be the corresponding $2\ZZ$-derived
  contraction algebra so that $\dgH[0]{\DerCon}\cong\Con=\Con*(p)$ is the
  contraction algebra defined by Donovan and Wemyss. For simplicity, we treat
  the isomorphism of graded algebras
  \[
    \dgH{\DerCon}\cong\Con[\imath^{\pm1}]=\Con\otimes_\CC\CC[\imath^{\pm1}],\qquad|\imath|=-2,
  \]
  as an identification.
\end{setting}

Kadeishvili's Homotopy Transfer Theorem~\cite{Kad82} provides us with a
minimal $A_\infty$-algebra structure, unique up to
$A_\infty$-isomorphism with identity linear part,
\[
  B=(\Lambda[\imath^{\pm1}],m_3,m_4,m_5,\cdots)
\]
on the cohomology algebra $\Lambda[\imath^{\pm1}]$. Since
$\Lambda[\imath^{\pm1}]$ is concentrated in even degrees and, by definition,
\[
  m_n\colon \Lambda[\imath^{\pm1}]^{\otimes n}\longrightarrow
  \Lambda[\imath^{\pm1}]
\]
is a morphism of degree $2-n$, we conclude that $m_{n}=0$ whenever $n$ is odd.
We write
\[
  B=(\Lambda[\imath^{\pm1}],m_4,m_6,m_8,\cdots)
\]
as a way to record this observation. We refer to $B$ as a \emph{minimal
  ($A_\infty$-algebra) model} of the DG algebra $\DerCon$ and fix it for the
rest of the section.

\begin{remark}
  \label{rmk:DG_vs_A-infinity}
  The passage from DG to $A_\infty$-algebras is a matter of technical
  convenience: The homotopy theories of \emph{non-unital} DG and of $A_\infty$-algebras are
  equivalent, \cite[Thm.~11.4.8]{LV12}. In particular, two \emph{non-unital} DG algebras are quasi-isomorphic if and only if their minimal models are $A_\infty$-isomorphic \cite[Thms.~11.4.9 and 10.3.10]{LV12}. Here, we are exclusively interested in \emph{unital} DG algebras, but this is not a problem since by \cite[Prop.~6.2]{Mur14} two \emph{unital} DG algebras are quasi-isomorphic as \emph{non-unital} DG algebras if and only if they are quasi-isomorphic as \emph{unital} DG algebras.
\end{remark}

Consider now the bigraded Hochschild cochain complex
\[
  \HC[p][q]{\Lambda[\imath^{\pm1}]}[\Lambda[\imath^{\pm1}]]\coloneqq\Hom[\CC]{\Lambda[\imath^{\pm1}]^{\otimes
      p}}{\Lambda[\imath^{\pm1}](q)},\qquad p\geq 0,\ q\in\ZZ,
\]
recalled in \Cref{subsec:HH}. Since $m_3=0$, the $A_\infty$-equations imply that $\partial(m_4)=0$
(\cite[Lemme~B.4.1]{Lef03}); hence we obtain a class
\begin{equation}
  \class{m_4}=\class{m_4^{\DerCon}}\in\HH[4][-2]{\Lambda[\imath^{\pm1}]}[\Lambda[\imath^{\pm1}]]
\end{equation}
that we call the \emph{universal Massey product (of length $4$)}. It follows from
the definition of $A_\infty$-morphism (\cite[Lemme B.4.2]{Lef03}) that the class
$\class{m_4}$ does not depend on the choice of minimal model for $\DerCon$ and
hence the universal Massey product can and will be regarded as an invariant of
the latter DG algebra.

Consider now the graded-algebra morphism
$j\colon\Lambda\hookrightarrow\Lambda[\imath^{\pm1}]$ given by the inclusion of
the degree $0$ part. The morphism $j$ induces a restriction morphism on
Hochschild cohomology\footnote{In fact, the morphism $j^*$ is surjective,
  see~\cite[Prop.~4.6.9]{JKM22} and take $\sigma=\id$ and $d=2$ (which is even
  and hence no signs occur in the formulas therein).}
\[
  j^*\!\colon\HH{\Lambda[\imath^{\pm1}]}[\Lambda[\imath^{\pm1}]]\longrightarrow\HH{\Lambda}[\Lambda[\imath^{\pm1}]],
\]
where the space on the right is the Hochschild cohomology of $\Lambda$, viewed
as a graded algebra concentrated in degree $0$, with coefficients in the graded
$\Lambda$-bimodule $\Lambda[\imath^{\pm1}]$. In particular, since the degree
$-2$ component $\Lambda\cdot\imath$ of $\Lambda[\imath^{\pm1}]$ is isomorphic to
the diagonal $\Lambda$-bimodule, we obtain a class
\begin{equation}
  j^*\!\class{m_4}=j^*\!\class{m_4^{\DerCon}}\in\HH[4][-2]{\Lambda}[\Lambda[\imath^{\pm1}]]=\HH*[4]{\Lambda}[\Lambda\cdot\imath]=\Ext[\Lambda^e]{\Lambda}{\Lambda}[4],
\end{equation}
where $\Lambda^e=\Lambda\otimes_\CC\Lambda^\op$ is the eveloping algebra of
$\Lambda$; we call the class $j^*\!\class{m_4}$ the \emph{restricted universal
  Massey product (of length $4$)} and, as with the unrestricted version, we regard
it as an invariant of the $2\ZZ$-derived contraction algebra $\DerCon$. Notice also
that the previous discussion applies verbatim to any DG algebra $A$ whose
cohomology is isomorphic to the graded algebra $\Lambda[\imath^{\pm1}]$, so that
we may associate to $A$ its restricted universal Massey product
$j^*\!\class{m_4^A}$.

The following theorem is the first main step towards the proof
of~\Cref{thm:JM-con}.

\begin{theorem}
  \label{thm:JM-rUMP_unit}
  The restricted universal Massey product $j^*\!\class{m_4}$, when viewed as an
  element of the space $\Ext[\Lambda^e]{\Lambda}{\Lambda}[4]$ of Yoneda
  extensions of $\Lambda$-bimodules, can be represented by an exact sequence
  \[
    0\to\Lambda\to P_3\to P_2\to P_1\to P_0\to\Lambda\to 0
  \]
  with projective middle terms. In particular, ${\Omega_{\Lambda^e}^4(\Lambda)\cong\Lambda}$ in the
  stable category of $\Lambda$-bimodules.
\end{theorem}
\begin{proof}
  The first claim is a special case of~\cite[Cor.~4.5.17]{JKM22}. Indeed, by
  definition, the $2\ZZ$-derived contraction algebra $\DerCon$ is the derived
  endomorphism algebra of a $2\ZZ$-cluster tilting object in
  $\DerCat[c]{\DerCon}\simeq\SingCat{R}$, which is one of the equivalent
  conditions in~\emph{loc.~cit.} The second claim follows immediately from the
  first.
\end{proof}

\begin{remark}
  \label{rmk:thm:JM-rUMP_unit}
  The proof of~\cite[Cor.~4.5.17]{JKM22}, and hence that of
  \Cref{thm:JM-rUMP_unit}, is non-trivial. In the special case of the
  contraction algebra, it is possible that detailed knowledge of the first
  non-trivial higher operation $m_4$ of some minimal model of the $2\ZZ$-derived
  contraction algebra allows for establishig the desired property of the
  restricted universal Massey product $j^*\!\class{m_4}$ directly. The approach
  taken in~\cite{JKM22}, which deals with an abstract and more general
  situation, rather leverages the fact that $\Lambda$ is the endomorphism
  algebra of a $2\ZZ$-cluster tilting object $T\in\SingCat{R}$. The upshot is
  that the additive closure $\add*{T}$ of $T$ has an induced structure of a
  so-called \emph{$4$-angulated category}, that is $\add*{T}$ is equipped with a
  natural class of diagrams $\square_{\mathrm{GKO}}$, called \emph{$4$-angles},
  of the form\footnote{Recall that $[2]\cong\id$ in $\SingCat{R}$.}
  \[
    T_1\to T_2\to T_3\to T_4\to T_1[2]
  \]
  that satisfies axioms analogous to those of triangulated
  categories~\cite{GKO13}. On the other hand, an extension of
  $\Lambda$-bimodules
  \[
    0\to\Lambda\to P_3\to P_2\to P_1\to P_0\to\Lambda\to 0
  \]
  with $P_0,P_1,P_2$ projective-injective (but perhaps not $P_3$) that
  represents the class $j^*\!\class{m_4}\in\Ext[\Lambda^e]{\Lambda}{\Lambda}[4]$
  yields a class of $4$-angles $\square_{j^*\!\class{m_4}}$ defined in terms of
  certain exactness properties~\cite{Ami07,Lin19}; the class
  $\square_{j^*\!\class{m_4}}$ is \emph{a priori} not known to form a
  $4$-angulation of $\add*{T}$. The crux of the argument is then to prove that
  \[
    \square_{\mathrm{GKO}}=\square_{j^*\!\class{m_4}}
  \]
  so that the class $\square_{j^*\!\class{m_4}}$ is indeed a $4$-angulation of
  $\add*{T}$; this agreement relies on a delicate analysis of the relationship
  between Toda brackets, Massey products and the classes
  $\square_{\mathrm{GKO}}$ and $\square_{j^*\!\class{m_4}}$. Finally, in view of
  the exactness properties defining the class $\square_{j^*\!\class{m_4}}$
  (now known to be $4$-angulation), a
  theorem of Auslander and Reiten~\cite{AR91a} for detecting projective
  bimodules implies that $P_3$ must be a projective $\Lambda$-bimodule, which is
  what \Cref{thm:JM-rUMP_unit} claims. The reader is referred to~\cite{JKM22}
  for details.
\end{remark}

Recall that the contraction algebra is Frobenius (in fact, symmetric).
Consequently, its enveloping algebra is also a Frobenius algebra and we may
consider the \emph{Hochschild--Tate cohomology}
\[
  \TateHH{\Lambda}[\Lambda[\imath^{\pm1}]]\coloneqq\sExt[\Lambda^e]{\Lambda}{\Lambda[\imath^{\pm1}]}[\bullet,*]
\]
defined in terms of the extension spaces in the stable category of graded
$\Lambda$-bimodules; thus,
\[
  \HH[>0]{\Lambda}[\Lambda[\imath^{\pm1}]]=\TateHH[>0]{\Lambda}[\Lambda[\imath^{\pm1}]]
\]
and there is a surjection
\[
  \HH[0]{\Lambda}[\Lambda[\imath^{\pm1}]]\twoheadrightarrow\TateHH[0]{\Lambda}[\Lambda[\imath^{\pm1}]].
\]
The multiplication on $\Lambda[\imath^{\pm1}]$ endows
$\TateHH{\Lambda}[\Lambda[\imath^{\pm1}]]$ with the structure of a bigraded
algebra, see~\cite[Sec.~5]{Mur22} for details.

\begin{corollary}
  \label{cor:JM-rUMP_unit}
  The restricted universal Massey product $j^*\!\class{m_4}$, when viewed as an
  element of the Hochschild--Tate cohomology
  $\TateHH{\Lambda}[\Lambda[\imath^{\pm1}]]$, is a unit.
\end{corollary}
\begin{proof}
  Immediate from~\Cref{thm:JM-rUMP_unit} and~\cite[Prop.~5.7 and
    Rmk.~5.8]{Mur22}, which characterises the units in
  $\TateHH{\Lambda}[\Lambda[\imath^{\pm1}]]$ of positive Hochschild
  (=horizontal) degree.
\end{proof}

\begin{remark}
  In~\Cref{cor:JM-rUMP_unit} it is essential to pass from Hochschild to
  Hochschild--Tate cohomology in order to have units of positive Hochschild
  degree.
\end{remark}

\subsection{Hochschild cohomology of the graded contraction algebra}
\label{subsec:HH_Con}

In this section we compute the Hochschild cohomology of the graded algebra
\[\Con*[\imath^{\pm1}]=\Lambda[\imath^{\pm1}]=\Lambda\otimes_\CC\CC[\imath^{\pm1}],\qquad \abs{\imath}=-2,\]
that we call \emph{graded contraction algebra}, in terms of the Hochschild
cohomology of the Dononvan--Wemyss contraction algebra $\Con*=\Lambda$.

First, notice that $\imath$ lies in the (graded) centre of
$\Lambda[\imath^{\pm1}]$, which is
\[
  Z(\Lambda[\imath^{\pm1}])=\HH[0]{\Lambda[\imath^{\pm1}]}[\Lambda[\imath^{\pm1}]];
\]
hence
\[\imath\in \HH[0][-2]{\Lambda[\imath^{\pm1}]}[\Lambda[\imath^{\pm1}]].\]
We introduce the \emph{fractional Euler derivation}
\[\EulerDer{2}\in\HC[1][0]{\Lambda[\imath^{\pm1}]}[\Lambda[\imath^{\pm1}]],\]
which acts by the formula
\[
  \EulerDer{2}\colon a\longmapsto\tfrac{|a|}{2}a,
\]
where we observe that $\frac{|a|}{2}$ is an integer since $\Lambda[\imath^{\pm1}]$
is concentrated in even degrees. It is a cocycle with cohomology class
\[\EulerClass{2}\in\HH[1][0]{\Lambda[\imath^{\pm1}]}[\Lambda[\imath^{\pm1}]].\]

\begin{proposition}\label{prop:HH_isos}
  The following statements hold:
  \begin{enumerate}
    \item\label{it:bracket} There is an isomorphism of graded commutative algebras
          \[
            \HH{\Lambda[\imath^{\pm1}]}[\Lambda[\imath^{\pm1}]]\cong\uHH{\Lambda}[\Lambda][\imath^{\pm1},\EulerClass{2}].
          \]
          The graded Lie algebra structure on the right hand side is induced by the
          (usual) Lie algebra structure on $\HH{\Lambda}[\Lambda]$ by setting
          \begin{align*}
            [\imath, \uHH{\Lambda}[\Lambda]]         & =0, & [\imath,\imath]         & =0,       \\
            [\EulerClass{2}, \uHH{\Lambda}[\Lambda]] & =0, & [\EulerClass{2},\imath] & =-\imath.
          \end{align*}
    \item\label{it:HH_iso} There is an isomorphism of graded algebras
          \[
            \HH{\Lambda}[\Lambda[\imath^{\pm1}]]\cong \uHH{\Lambda}[\Lambda][\imath^{\pm1}].
          \]
          Moreover, the morphism
          \[
            j^*\colon\HH{\Lambda[\imath^{\pm1}]}[\Lambda[\imath^{\pm1}]]\longrightarrow\HH{\Lambda}[\Lambda[\imath^{\pm1}]]
          \]
          induced by the inclusion $j\colon \Lambda\hookrightarrow
            \Lambda[\imath^{\pm1}]$ of the degree $0$ part is the apparent natural
          projection with kernel the graded ideal generated by $\EulerClass{2}$.
    \item\label{it:TateHH_iso} There is an isomorphism of graded algebras
          \[
            \TateHH{\Lambda}[\Lambda[\imath^{\pm1}]]\cong \uTateHH{\Lambda}[\Lambda]
            [\imath^{\pm1}].
          \]
          Furthermore, the comparison map
          \[
            \HH{\Lambda}[\Lambda[\imath^{\pm1}]]\longrightarrow\TateHH{\Lambda}[\Lambda[\imath^{\pm1}]]
          \]
          is the apparent extension of the comparison map $\uHH{\Lambda}[\Lambda] \to
            \uTateHH{\Lambda}[\Lambda]$.
  \end{enumerate}
\end{proposition}

\begin{proof}
  All of the forthcoming claims follow from the proof of
  \cite[Prop.~4.6.9]{JKM22} for $\sigma=\id[\Lambda]$ and $d=2$.

  \eqref{it:bracket} The Hochschild
  complex $\HC{\Lambda[\imath^{\pm1}]}[\Lambda[\imath^{\pm1}]]$ contains the
  subcomplex
  \[\RHC{\Lambda[\imath^{\pm1}]}[\Lambda[\imath^{\pm1}]]\]
  of $\CC[\imath^{\pm1}]$-linear cochains; this subcomplex is also an
  associative subalgebra and a Lie subalgebra of the $\CC$-linear Hochschild
  complex.

  The composite
  \[\RHC{\Lambda[\imath^{\pm1}]}[\Lambda[\imath^{\pm1}]]\stackrel{i}{\hookrightarrow}
    \HC{\Lambda[\imath^{\pm1}]}[\Lambda[\imath^{\pm1}]]\stackrel{j^*}{\longrightarrow}
    \HC{\Lambda}[\Lambda[\imath^{\pm1}]],\] of the inclusion of the
  $\CC[\imath^{\pm1}]$-linear Hochschild cochains into the $\CC$-linear ones with
  the restriction of scalars along the inclusion $j\colon\Lambda\hookrightarrow
    \Lambda[\imath^{\pm1}]$ of the degree $0$ part is an \emph{isomorphism} of DG
  algebras. The target, unlike the source, does not \emph{a priori} carry any
  Lie algebra structure. Nevertheless, there is an obvious isomorphism of DG
  algebras
  \[
    \HC{\Lambda}[\Lambda[\imath^{\pm1}]]\cong\uHC{\Lambda}[\Lambda][\imath^{\pm1}]
  \]
  that we regard as an identification, and the composite isomorphism
  \[
    \RHC{\Lambda[\imath^{\pm1}]}[\Lambda[\imath^{\pm1}]]\cong\uHC{\Lambda}[\Lambda][\imath^{\pm1}]
  \]
  is also a Lie algebra map when we regard the target as a Lie algebra extension
  of $\uHC{\Lambda}[\Lambda]$ with $\imath$ a central element (in the
  Lie-algebra sense).

  The morphism
  \begin{align*}
    \HC{\Lambda[\imath^{\pm1}]}[\Lambda[\imath^{\pm1}]] & \longrightarrow		\uHC{\Lambda}[\Lambda] [\imath^{\pm1},\EulerDer{2}],  \\
    x                                                   & \longmapsto j^*(x)-\imath^{-1}\cdot j^*([x,\imath])\cdot \EulerDer{2},
  \end{align*}
  is a quasi-isomorphism of DG algebras with quasi-inverse
  \begin{align*}
    \uHC{\Lambda}[\Lambda] [\imath^{\pm1}]\oplus \uHC{\Lambda}[\Lambda] [\imath^{\pm1}]\cdot \EulerDer{2} & \longrightarrow			\HC{\Lambda[\imath^{\pm1}]}[\Lambda[\imath^{\pm1}]], \\
    x+y\cdot\EulerDer{2}                                                                                  & \longmapsto						i(x)+i(y)\cdot\EulerDer{2}.
  \end{align*}
  (The latter is just a morphism of complexes since $\EulerDer{2}^2\neq0$ in the
  target, it only vanishes in cohomology.) In fact, the relevant composite equals the
  identity of $\uHC{\Lambda}[\Lambda] [\imath^{\pm1},\EulerDer{2}]$. The Lie
  bracket formulas in the statement of the proposition follow from the
  definition of the fractional Euler class, $\CC[\imath^{\pm1}]$-linear cochains
  and degree considerations.

  \eqref{it:HH_iso} The statement follows easily from the previous
  computations.

  \eqref{it:TateHH_iso} The statement is consequece of the fact that
  $\TateHH{\Lambda}[\Lambda]$ is obtained from
  $\HH{\Lambda}[\Lambda]$ by inverting any element of
  \[
    \uHH[4]{\Lambda}[\Lambda]=\Ext[\Lambda^e]{\Lambda}{\Lambda}[4]
  \]
  representing the $4$-periodicty of $\Lambda$, and
  similarly when the coefficients lie in $\Lambda[\imath^{\pm1}]$.
\end{proof}

Below, we use the isomorphisms in \Cref{prop:HH_isos} as identifications.

\begin{corollary}\label{cor:two_equations_determine}
  Let $u\in\uTateHH[4]{\Lambda}[\Lambda]$ be a unit. There exists a unique
  Hochschild class
  \[
    m\in\HH[4][-2]{\Lambda[\imath^{\pm1}]}[\Lambda[\imath^{\pm1}]]
  \]
  such that
  \begin{align*}
    j^*(m)            & =u\cdot\imath, &
    \tfrac{1}{2}[m,m] & =0.
  \end{align*}
\end{corollary}

\begin{proof}
  The first equation in the statement is equivalent to $m$ being of the form
  \begin{equation}
    \label{eq:m_in_terms_of_u}
    m=(u + x\cdot \EulerClass{2})\cdot \imath
  \end{equation}
  for some $x\in\uHH[3]{\Lambda}[\Lambda]$. Using the relations in a Gerstenhaber algebra, the second equation is equivalent to
  \begin{align*}
    0
     & = ([u,u]-2u\cdot x) \cdot \imath^2
    -2[x,u] \cdot \imath^2\cdot \EulerClass{2}.
  \end{align*}
  This means that both summands must vanish. For the first one, this is equivalent to
  \[x=\tfrac{1}{2}u^{-1}[u,u].\]
  This takes place in the piece of $\uHH{\Lambda}[\Lambda]$ that agrees with $\uTateHH{\Lambda}[\Lambda]$, and is compatible with the second summand since
  \begin{align*}
    0 & = \tfrac{1}{2}[[u,u],u]=[u x,u]=u [x,u]-[u,u] x=u [x,u]-u^{-1}[u,u]^2=u [x,u],
  \end{align*}
  so $[x,u]=0$. The first step follows from the graded Jacobi identity and we also
  use that $[u,u]^2=0$ since $[u,u]$ has odd total degree and Hochschild
  cohomology is graded commutative.
\end{proof}

The following result should be compared with equation
\eqref{eq:m_in_terms_of_u}; its proof is similar to that of \Cref{cor:two_equations_determine}.

\begin{corollary}\label{lem:contractible}
  Let $u\in\uTateHH[4]{\Lambda}[\Lambda]$ be a unit such that $[u,u]=0$. Given $$(x+y\cdot\EulerClass{2} )\cdot\imath^q\in\HH[p][-2q]{\Lambda[\imath^{\pm1}]}[\Lambda[\imath^{\pm1}]]$$
  with $p\geq 2$, $x\in\uHH[p]{\Lambda}[\Lambda]$ and $y\in\uHH[p-1]{\Lambda}[\Lambda]$, if $[u\cdot\imath, (x+y\cdot\EulerClass{2})\cdot\imath^q]=0$ then
  $$(x+y\cdot\EulerClass{2})\cdot\imath^q =[u\cdot\imath,u^{-1} \cdot\EulerClass{2} \cdot x \cdot\imath^{q-1}].$$
\end{corollary}

We obtain the following more precise information on a minimal $A_\infty$-model
of the $2\ZZ$-derived contraction algebra.

\begin{proposition}\label{prop:periodicity}
  The $2\ZZ$-derived contraction algebra has a minimal $A_\infty$-model
  \[
    (\Lambda[\imath^{\pm1}],m_4,m_6,\cdots)
  \]
  such that $m_n$ is $\CC[\imath^{\pm1}]$-linear for all $n\geq 4$. In
  particular, $\{m_4\}=u\cdot\imath$ for some unit
  $u\in\uTateHH[4]{\Lambda}[\Lambda]$ satisfying $[u,u]=0$.
\end{proposition}
\begin{proof}
  The first part follows from \cite{HK18}. The rest is a direct consecuence of \Cref{prop:HH_isos} and the fact that $[\{m_4\},\{m_4\}]=0$, which follows from \eqref{eq:MC}.
\end{proof}

\subsection{Proof of \Cref{thm:JM-con}}
\label{sec:the_proof}

The introduction of the restricted universal Massey product of $\DerCon$ is
justified by the following result and the subsequent corollary.
\Cref{thm:rUMP_ThmB} is an immediate consequence of~\cite[Thm.~B]{JKM22}, and
the latter theorem is obtained as an application of the obstruction theory for the
existence of $A_\infty$-structures developed by the third-named author
in~\cite{Mur20b}. In this note we give a direct proof of \Cref{thm:rUMP_ThmB}
that leverages our detailed knowledge of the relationship between the Hochschild
cohomology of the contraction algebra and that of its graded variant
(see~\Cref{subsec:HH_Con}), although part of the techniques used to
prove~\cite[Thm.~B]{JKM22} are utilised in some guise.

\begin{theorem}
  \label{thm:rUMP_ThmB}
  Let $A$ be a DG algebra such that $\dgH{A}=\Lambda[\imath^{\pm1}]$ as graded
  algebras. If
  \[
    j^*\!\class{m_4^A}=j^*\!\class{m_4^{\DerCon}}\in\TateHH{\Lambda}[\Lambda[\imath^{\pm1}]],
  \]
  then $A$ is quasi-isomorphic to the $2\ZZ$-derived contraction algebra $\DerCon$ via
  a quasi-iso\-mor\-phism that induces the identity in cohomology.
\end{theorem}

\begin{proof}
  Let
  \[(\Lambda[\imath^{\pm1}],m_4,m_6,\dots)\]
  be a minimal model for the $2\ZZ$-derived contraction algebra as in \Cref{prop:periodicity}, and
  \[(\Lambda[\imath^{\pm1}],m_4',m_6',\dots)\]
  a minimal model for $A$. Inductively, we will construct
  an $A_\infty$-isomorphism with identity linear part
  \[f=(0,0,0,f_3,0,f_5,\dots)\colon
    (\Lambda[\imath^{\pm1}],m_4,m_6,\dots)\longrightarrow
    (\Lambda[\imath^{\pm1}],m_4',m_6',\dots),\] and this clearly suffices to
  prove the claim. Notice that, necessarily, $f_{2n}=0$ for all $n\geq 0$ since
  $\Lambda[\imath^{\pm1}]$ is concentrated in even degrees.

  We proceed as follows. For all $n\geq 0$ we define a Hochschild cochain of
  total degree $1$
  \[
    f^{(n)}=(0,0,0,f^{(n)}_3,0,\dots, f^{(n)}_{2n+1},0,\dots)
  \]
  such that $f^{(n)}$ coincides with $f^{(n-1)}$ up to Hochschild degree $2n-2$
  and \addtocounter{equation}{1}
  \begin{equation}\tag*{(\theequation)$^{\alto}_{\bajo}$}\label{oso}
    \partial(f^{(n)}) +f^{(n)}\cdot f^{(n)}+\sum_{r\geq 0}m'\{f^{(n)},\stackrel{r}{\dots},f^{(n)}\}-m-f^{(n)}\{m\}
  \end{equation}
  vanishes up to Hochschild degree $2n+2$. If we achieve this goal, then we can
  take
  \[f=(0,0,0,f_3^{(3)},0,\dots, f_{2n-3}^{(n)},0,\dots).\]
  Indeed, $f$ coincides with $f^{(n)}$ up to Hochschild degree $2n-2$, so \eqref{A-map}
  coincides with \ref{oso} up to Hochschild degree $2n-1$. In particular \eqref{A-map} vanishes up to Hochschild degree $2n-1$ for all $n\geq 0$. Therefore \eqref{A-map} fully vanishes, so $f$ is indeed an $A_\infty$-isomorphism with identity linear part.

  We start with $f^{(0)}=0$. With this choice, {\renewcommand\alto{0}\ref{oso}} clearly vanishes up to Hochschild degree $2$.

  Below, when defining $f^{(n)}$ we will only specify $f^{(n)}_{2n-1}$ and $f^{(n)}_{2n+1}$ since in smaller Hochschild degrees they are determined by $f^{(n-1)}$ and in higher Hochschild degrees they are irrelevant. Moreover, we will also use that {\renewcommand\alto{n-1}\ref{oso}} and {\renewcommand\alto{n}\ref{oso}} agree (and hence both vanish) up to Hochschild degree $2n-1$.

  Since $j^*\{m_4\}=j^*\{m_4'\}$, then $\{m_4\}=\{m_4'\}$ by \Cref{cor:two_equations_determine}, so there exists $f_3^{(1)}$ such that
  \begin{equation}\label{A-map-4}
    \partial(f_3^{(1)})+m_4'-m_4=0.
  \end{equation}
  This proves that {\renewcommand\alto{1}\ref{oso}} vanishes up to Hochschild degree $4$.

  Assume we have constructed up to $f^{(n-1)}$ for some $n\geq 2$. Let us see how to construct $f^{(n)}$.
  We know by \cite[Lemme B.4.2]{Lef03} that the Hochschild degree $2n+2$ part of {\renewcommand\alto{n-1}\ref{oso}}, that we simply denote by $a$, is an obstruction cocycle ($\partial(a)=0$) which vanishes in cohomology if and only if there exists $f_{2n+1}^{(n)}$ such that, taking $f_{2n-1}^{(n)}=f_{2n-1}^{(n-1)}$, {\renewcommand\alto{n}\ref{oso}} vanishes up to Hochschild degree $2n+2$. We claim that
  \begin{equation}\label{claim}
    [m_4,a]+\partial(b-f_3^{(n-1)}\{a\})
  \end{equation}
  vanishes, where $b$ is the Hochschild degree $2n+4$ part of {\renewcommand\alto{n-1}\ref{oso}}. We prove this claim below. Now, we deduce from \Cref{prop:periodicity} and \Cref{lem:contractible} that there exist Hochschild cochains $c_{2n-1}$ and $c_{2n+1}$ such that
  \begin{align*}
    a+\partial(c_{2n+1})+[m_4,c_{2n-1}] & =0, & \partial(c_{2n-1}) & =0.
  \end{align*}
  If we set
  \begin{align*}
    f^{(n)}_{2n-1} & = f^{(n-1)}_{2n-1}+c_{2n-1}, &
    f^{(n)}_{2n+1} & =
    \left\{\begin{array}{ll}
             c_5+ f^{(1)}_3\{c_3\}+\frac{1}{2}c_3\{c_3\}, & n=2; \\[2mm]
             c_{2n+1}+f_3^{(n-1)}\{c_{2n-1}\},                    & n>2;
           \end{array}\right.
  \end{align*}
  we complete the induction step since the Hochschild degree $2n$ part of {\renewcommand\alto{n}\ref{oso}} is,
  \[\partial(c_{2n-1})=0,\]
  and its Hochschild degree $2n+2$ part is, for $n=2$,
  \begin{multline*}
    a+\partial\Big(c_5+ f^{(1)}_3\{c_3\}+\frac{1}{2}c_3\{c_3\}\Big)+ f_3^{(1)}\cdot c_3 + c_3\cdot f_3^{(1)}+c_{3}\cdot c_{3} +m_4'\{c_{3}\}-c_{3}\{m_4\}\\=
    a+\partial(c_{5})+\partial\big(f^{(1)}_3\big)\{c_3\}+ f^{(1)}_3\{\partial(c_3)\} +m_4'\{c_{3}\}-c_{3}\{m_4\}\\
    = a+\partial(c_{5})+(m_4-m_4')\{c_3\}+m_4'\{c_{3}\}-c_{3}\{m_4\}\\
    = a+\partial(c_{5})+[m_4,c_3]=0,
  \end{multline*}
  where we use that $\partial(c_3\{c_3\})=-2c_3\cdot c_3$ by \Cref{braces_and_differential}, and for $n>2$,
  \begin{multline*}
    a+\partial\Big(c_{2n+1}+f_3^{(n-1)}\{c_{2n-1}\}\Big)+ f_3^{(n-1)}\cdot c_{2n-1} + c_{2n-1}\cdot f_3^{(n-1)}+m_4'\{c_{2n-1}\}-c_{2n-1}\{m_4\}\\=
    a+\partial(c_{2n+1}) +\partial\big(f^{(n-1)}_3\big)\{c_{2n-1}\}+ f^{(n-1)}_3\{\partial(c_{2n-1})\} +m_4'\{c_{2n-1}\}-c_{2n-1}\{m_4\}\\
    =
    a+\partial(c_{2n+1}) +(m_{4}-m_{4}')\{c_{2n-1}\}+m_4'\{c_{2n-1}\}-c_{2n-1}\{m_4\}\\
    =
    a+\partial(c_{2n+1}) +[m_4,c_{2n-1}]=0.
  \end{multline*}

  We finish the proof with the vanishing of \eqref{claim}. In what follows, let us write $\Xi=\text{{\renewcommand\alto{n-1}\ref{oso}}}$ and $f=f^{(n-1)}$, so as not to overload notation. Note that \eqref{claim} is the Hochschild degree $2n+5$ part of
  \begin{equation}\label{claim2}
    [m,\Xi]+\partial(\Xi-f\{\Xi\}).
  \end{equation}
  This cochain vanishes in Hochschild degrees $<2n+5$.

  We now start a series of computations. We number most terms for bookkeeping purposes. In the first equation we use \Cref{braces_and_differential},
  \begin{align*}
    \partial(\Xi)= & \ \overbracket{\partial(f)\cdot f}^{\mytag{d(f)f}}-\overbracket{f\cdot\partial{f}}^{\mytag{fd(f)}}
    +\overbracket{\sum_{r\geq0}\partial(m')\{f,\stackrel{r}{\dots},f\}}^{\mytag{d(m')(f..f)}}                                                           \\
                   & -\overbracket{\sum_{r\geq1}\sum_{i=1}^rm'\{f,\stackrel{i-1}{\dots},\partial(f),\stackrel{r-i}{\dots},f\}}^{\mytag{m'(f..d(f)..f)}} \\
                   & -\overbracket{\sum_{r\geq1}f\cdot m'\{f,\stackrel{r-1}{\dots},f\}}^{\mytag{fm'(f..f)}}
    -\overbracket{\sum_{r\geq2}\sum_{i=1}^{r-1}m'\{f,\stackrel{i-1}{\dots},f^2,\stackrel{r-i-1}{\dots},f\}}^{\mytag{m'(f..f2..f)}}
    \\
                   & + \overbracket{\sum_{r\geq1}m'\{f,\stackrel{r-1}{\dots},f\}\cdot f}^{\mytag{m'(f..f)f}}-\overbracket{\partial(m)}^{\mytag{d(m)}}
    -\overbracket{\partial(f)\{m\}}^{\mytag{d(f)(m)}}
    -\overbracket{f\{\partial(m)\}}^{\mytag{f(d(m))}}
    +\overbracket{f\cdot m}^{\mytag{fm}}
    -\overbracket{m\cdot f}^{\mytag{mf}}
  \end{align*}
  Since $m$ and $m'$ are $A_\infty$-algebra structures,
  \begin{align*}
    \text{\refcirc{d(m)}}        & =-m\{m\},                                                                                                                     \\
    \text{\refcirc{f(d(m))}}     & =-f\{m\{m\}\},                                                                                                                \\
    \text{\refcirc{d(m')(f..f)}} & =-\sum_{r\geq 0}m'\{m'\}\{f,\stackrel{r}{\dots},f\}                                                                           \\
                                 & =-\sum_{r\geq0}\sum_{0\leq i\leq j\leq r}m'\{f,\stackrel{i}{\dots},m'\{f,\stackrel{j-i}{\dots},f\},\stackrel{r-j}{\dots},f\}, \\
                                 & =-\overbracket{\sum_{r\geq0}m'\{m'\{f,\stackrel{r}{\dots},f\}\}}^{\mytag{m'(m'(f..f))}}
    -\overbracket{\sum_{r\geq1}\sum_{\substack{0\leq i\leq j\leq r                                                                                               \\j-i<r}}m'\{f,\stackrel{i}{\dots},m'\{f ,\stackrel{j-i}{\dots},f\} ,\stackrel{r-j}{\dots}, f\}}^{\mytag{m'(f..m'(f..f)..f)_one_f_out}}
  \end{align*}
  Here we also use the brace relation.
  We also split the following summations in two parts,
  \begin{align*}
    \text{\refcirc{m'(f..d(f)..f)}} & =\overbracket{m'\{\partial(f)\}}^{\mytag{m'(d(f))}}+ \overbracket{\sum_{r\geq2}\sum_{i=1}^rm'\{f,\stackrel{i-1}{\dots},\partial(f),\stackrel{r-i}{\dots},f\}}^{\mytag{m'(f..d(f)..f)_r>=2}}, \\
    \text{\refcirc{m'(f..f2..f)}}   & =\overbracket{m'\{f^2\}}^{\mytag{m'(f2)}}
    + \overbracket{\sum_{r\geq3}\sum_{i=1}^{r-1}m'\{f,\stackrel{i-1}{\dots},f^2,\stackrel{r-i-1}{\dots},f\}}^{\mytag{m'(f..f2..f)_r>=3}}.
  \end{align*}
  Consider the following cochain, that we decompose using the brace relation,
  \begin{align*}
    \overbracket{\sum_{r\geq0}m'\{f,\stackrel{r}{\dots},f\}\{m\}}^{\mytag{m'(f..f)(m)}} & =
    \sum_{r\geq0}\sum_{i=0}^r m'\{f,\stackrel{i}{\dots},m ,\stackrel{r-i}{\dots}, f\}                                                                                                                                                   \\
                                                                                        & \phantom{=}+	\sum_{r\geq1}\sum_{i=0}^r m'\{f,\stackrel{i-1}{\dots},f\{m\} ,\stackrel{r-i}{\dots}, f\}                                         \\
                                                                                        & =
    \overbracket{m'\{m\}}^{\mytag{m'(m)}}+\overbracket{\sum_{r\geq1}\sum_{i=0}^r m'\{f,\stackrel{i}{\dots},m ,\stackrel{r-i}{\dots}, f\}}^{\mytag{m'(f..m..f)}}
    +\overbracket{m'\{f\{m\}\}}^{\mytag{m'(f(m))}}                                                                                                                                                                                      \\
                                                                                        & \phantom{=}+	\overbracket{\sum_{r\geq2}\sum_{i=0}^r m'\{f,\stackrel{i-1}{\dots},f\{m\} ,\stackrel{r-i}{\dots}, f\}}^{\mytag{m'(f..f(m)..f)}}.
  \end{align*}
  Consider also the following cochain, which is computed by using \Cref{braces_and_cup_product},
  \begin{align*}
    \overbracket{f^2\{m\}}^{\mytag{f2(m)}} & =\overbracket{f\cdot f\{m\}}^{\mytag{ff(m)}}-\overbracket{f\{m\}\cdot f}^{\mytag{f(m)f}}.
  \end{align*}
  Since $\Xi$ vanishes up to Hochschild degree $2n+1$ and its Hochschild degree $2n+2$ part is a cocycle, the following cochains vanish up to Hochschild degree $2n+5$,
  \begin{align*}
    \sum_{r\geq2}\sum_{i=1}^rm'\{f,\stackrel{i-1}{\dots},\Xi,\stackrel{r-i}{\dots},f\}=\text{\refcirc{m'(f..d(f)..f)_r>=2}}
    + \text{\refcirc{m'(f..f2..f)_r>=3}}
    + \text{\refcirc{m'(f..m'(f..f)..f)_one_f_out}}
    -\text{\refcirc{m'(f..m..f)}}
    -\text{\refcirc{m'(f..f(m)..f)}}, \\
    \partial(f)\{\Xi\}+m'\{\Xi\}-m\{\Xi\},\qquad f\{\partial(\Xi)\}.
  \end{align*}
  Notice that
  \begin{align*}
    \Xi\cdot f                  & =\text{\refcirc{d(f)f}}+f^3+ \text{\refcirc{m'(f..f)f}}-\text{\refcirc{mf}}-\text{\refcirc{f(m)f}}, &
    f\cdot\Xi                   & =\text{\refcirc{fd(f)}}+f^3+ \text{\refcirc{fm'(f..f)}}-\text{\refcirc{fm}}-\text{\refcirc{ff(m)}},   \\
    m'\{\Xi\}                   & =\text{\refcirc{m'(d(f))}}+ \text{\refcirc{m'(f2)}}
    + \text{\refcirc{m'(m'(f..f))}}
    -\text{\refcirc{m'(m)}}
    -\text{\refcirc{m'(f(m))}}, &
    \Xi\{m\}                    & = \text{\refcirc{d(f)(m)}}+ \text{\refcirc{f2(m)}}
    + \text{\refcirc{m'(f..f)(m)}}+\text{\refcirc{d(m)}}+\text{\refcirc{f(d(m))}}.
  \end{align*}

  Using all the above, we obtain the following relations, where $\equiv$ stands for congruence modulo cochains vanishing in Hochschild degrees $\leq 2n+5$,
  \begin{align*}
    \eqref{claim2}={} & m\{\Xi\}+\Xi\{m\}+\partial(\Xi)-\partial(f)\{\Xi\}-f\{\partial(\Xi)\}+f\cdot\Xi-\Xi\cdot f                                                                                                                                                                                                                                     \\
    \equiv{}          & m\{\Xi\}+\left(\text{\refcirc{d(f)(m)}}+ \text{\refcirc{f2(m)}}
    + \text{\refcirc{m'(f..f)(m)}}+\text{\refcirc{d(m)}}+\text{\refcirc{f(d(m))}}\right)                                                                                                                                                                                                                                                               \\
                      & +\left(\text{\refcirc{d(f)f}}-\text{\refcirc{fd(f)}}+ \text{\refcirc{d(m')(f..f)}}-\text{\refcirc{m'(f..d(f)..f)}}-\text{\refcirc{fm'(f..f)}}-\text{\refcirc{m'(f..f2..f)}}+ \text{\refcirc{m'(f..f)f}}-\text{\refcirc{d(m)}}-\text{\refcirc{d(f)(m)}}-\text{\refcirc{f(d(m))}}+\text{\refcirc{fm}}-\text{\refcirc{mf}}\right) \\
                      & +\left(\text{\refcirc{m'(d(f))}}+ \text{\refcirc{m'(f2)}}
    + \text{\refcirc{m'(m'(f..f))}}
    -\text{\refcirc{m'(m)}}
    -\text{\refcirc{m'(f(m))}}\right)-m\{\Xi\}-0                                                                                                                                                                                                                                                                                                       \\
                      & +\left(\text{\refcirc{fd(f)}}+f^3+ \text{\refcirc{fm'(f..f)}}-\text{\refcirc{fm}}-\text{\refcirc{ff(m)}}\right)
    -\left(\text{\refcirc{d(f)f}}+f^3+ \text{\refcirc{m'(f..f)f}}-\text{\refcirc{mf}}-\text{\refcirc{f(m)f}}\right)                                                                                                                                                                                                                                    \\
                      & + \left(\text{\refcirc{m'(f..d(f)..f)_r>=2}}
    + \text{\refcirc{m'(f..f2..f)_r>=3}}
    + \text{\refcirc{m'(f..m'(f..f)..f)_one_f_out}}
    -\text{\refcirc{m'(f..m..f)}}
    -\text{\refcirc{m'(f..f(m)..f)}}\right) =0.
  \end{align*}
  This finally concludes the proof.
\end{proof}

\begin{corollary}
  \label{coro:rUMP_determines}
  Let $A$ be a DG algebra such that $\dgH{A}=\Lambda[\imath^{\pm1}]$ as graded
  algebras. If the restricted universal Massey product
  \[
    j^*\!\class{m_4^A}\in\TateHH{\Lambda}[\Lambda[\imath^{\pm1}]]
  \]
  is a unit, then $A$ is quasi-isomorphic to the $2\ZZ$-derived contraction algebra
  $\DerCon$.
\end{corollary}
\begin{proof}
  Firstly, we observe that the group of graded-algebra automorphisms of
  $\Lambda[\imath^{\pm1}]$ acts on the right of
  $\HH{\Lambda}[\Lambda[\imath^{\pm1}]]$ by conjugation. In particular, the group
  $Z(\Lambda)^\times$ of units of the centre of $\Lambda$ acts on
  $\HH{\Lambda}[\Lambda[\imath^{\pm1}]]$ via the graded-algebra
  automorphisms
  \[
    g_u\colon x\longmapsto x u^{i},\qquad |x|=2i,
  \]
  where $u\in Z(\Lambda)^\times$. The induced action on
  \[
    \TateHH[4][-2]{\Lambda}[\Lambda[\imath^{\pm1}]]\cong\Ext[\Lambda^e]{\Lambda}{\Lambda}[4]
  \]
  has the following explicit description: Given a unit $u\in Z(\Lambda)^\times$
  and an exact sequence of $\Lambda$-bimodules
  \[
    \eta\colon 0\to\Lambda\xrightarrow{f} X_3\to X_2\to X_1\to X_0\to \Lambda\to
    0,
  \]
  we let $[\eta]\cdot u$ be the class of the exact sequence
  \[
    0\to\Lambda\xrightarrow{f'} X_3\to X_2\to X_1\to X_0\to \Lambda\to 0,
  \]
  where $f'\coloneqq u^{-1}f$.
  The above action clearly restricts to the set of units in
  $\TateHH{\Lambda}[\Lambda[\imath^{\pm1}]]$ of bidegree $(4,-2)$, since these are
  precisely the classes that can be represented by an exact sequence with
  projective(-injective) middle terms~\cite[Rmk.~5.8]{Mur22}, and is in fact
  transitive on this set. To prove the latter claim on the transitivity of the
  action we appeal to~\cite[Cor.~2.3]{Che21}, which allows us to lift stable bimodule
  isomorphisms $\Lambda\simeq\Lambda$ to honest bimodule isomorphisms
  $\Lambda\cong\Lambda$, see the proof of~\cite[Prop.~2.2.16]{JKM22}.

  Let $u\in Z(\Lambda)^\times$ be a unit such that
  \[
    j^*\!\class{m_4^A}\cdot
    u=j^*\!\class{m_4^{\DerCon}}\in\HH[4][-2]{\Lambda}[\Lambda[\imath^{\pm1}]].
  \]
  Given a minimal model
  \[
    (\Lambda[\imath^{\pm1}],m_4^A,m_6^A,m_8^A,\cdots)
  \]
  of the DG algebra $A$ we define a new minimal $A_\infty$-algebra structure
  \[
    (\Lambda[\imath^{\pm1}],\underline{m}_4^A,\underline{m}_6^A,\underline{m}_8^A,\cdots)=(\Lambda[\imath^{\pm1}],m_4^A,m_6^A,m_8^A,\cdots)*g_u
  \]
  with $n$-ary opearations $\underline{m}_n^A\coloneqq g_u^{-1}m_4^Ag_u^{\otimes
    n}$ (notice that $\underline{m}_4^A\neq m_4^A$ as soon as $u\neq1$). It is
  easy to see that
  \[
    j^*\!\class{\underline{m}_4^A}=j^*\!\class{m_4^A}\cdot
    u=j^*\!\class{m_4^{\DerCon}}
  \]
  and that there is an isomorphism of $A_\infty$-algebras
  \[
    (\Lambda[\imath^{\pm1}],m_4^A,m_6^A,m_8^A,\cdots)\rightsquigarrow(\Lambda[\imath^{\pm1}],\underline{m}_4^A,\underline{m}_6^A,\underline{m}_8^A,\cdots).
  \]
  Thus, $A$ is quasi-isomorphic to any DG algebra model $B$ of the minimal
  $A_\infty$-algebra on the right-hand side, see \Cref{rmk:DG_vs_A-infinity}, and by \Cref{thm:rUMP_ThmB} the DG
  algebras $B$ and $\DerCon$ are quasi-isomorphic. Consequently, the DG algebras
  $A$ and $\DerCon$ are quasi-isomorphic, which is what we needed to prove.
\end{proof}

We are ready to prove~\Cref{thm:JM-con}.
\begin{proof}[Proof of~\Cref{thm:JM-con}] Let $R_1$ and $R_2$ be isolated cDV
  singularities with crepant resolutions
  \[
    p_1\colon X_1\to \Spec*{R_1}\qquad\text{and}\qquad p_2\colon X_2\to
    \Spec*{R_2}
  \]
  whose corresponding contraction algebras $\Lambda_1=\Con*(p_1)$ and
  $\Lambda_2=\Con*(p_2)$ are isomorphic. We need to prove that the $2\ZZ$-derived
  contraction algebras $\DerCon_1=\DerCon*(p_1)$ and $\DerCon_2=\DerCon*(p_2)$
  are quasi-isomorphic. Recall that
  \[
    \dgH{\DerCon_1}\cong\Lambda_1[\imath^{\pm1}]\qquad\text{and}\qquad\dgH{\DerCon_2}\cong\Lambda_2[\imath^{\pm1}],\qquad |\imath|=-2.
  \]
  In view of the assumption that $\Lambda_1\cong\Lambda_2$, we obtain a chain of
  isomorphisms of graded algebras
  \[
    \dgH{\DerCon_1}\cong\Lambda_1[\imath^{\pm1}]\cong\Lambda_2[\imath^{\pm1}]\cong\dgH{\DerCon_2}.
  \]
  In particular, we may and we will identify the underlying graded algebra of minimal models of $\DerCon_1$ and
  $\DerCon_2$ via the above isomorphism. For simplicity, let
  $\Lambda=\Lambda_1\cong\Lambda_2$. \Cref{cor:JM-rUMP_unit} shows that the
  restricted universal Massey products
  \[
    j^*\!\class{m_4^{\DerCon_1}}\in\HH[4][-2]{\Lambda}[\Lambda[\imath^{\pm1}]]\qquad\text{and}\qquad
    j^*\!\class{m_4^{\DerCon_2}}\in\HH[4][-2]{\Lambda}[\Lambda[\imath^{\pm1}]]
  \]
  are units in the Hochschild--Tate cohomology
  $\TateHH{\Lambda}[\Lambda[\imath^{\pm1}]]$. \Cref{coro:rUMP_determines} implies that
  the DG algebras $\DerCon_1$ and $\DerCon_2$ are quasi-isomorphic, which is
  what we needed to prove.
\end{proof}

We also have the following important corollary.

\begin{corollary}
  \label{coro:dgSingCat_unique}
  Let $R$ be an isolated cDV singularity that admits a crepant resolution. Then,
  the singularity category $\SingCat{R}$ admits a unique DG enhancement in the
  sense of~\cite{BK90}.
\end{corollary}
\begin{proof}
  Let $\A$ be a DG enhancement of $\SingCat{R}$. By definition, this means that
  $\A$ is a pre-triangulated DG category and there exists an equivalence of
  triangulated categories
  \[
    \dgH[0]{\A}\simeq\SingCat{R}.
  \]
  Let $T\in\SingCat{R}$ be a $2\ZZ$-cluster tilting object and $A$ the
  derived endomorphism algebra of $T$ computed by means of the DG enhancement
  $\A$. In particular $\DerCat[c]{A}_{\mathrm{dg}}\simeq\A$ since $T$ is a classical generator. By definition, $\dgH{A}\cong\dgH{\DerCon}$ where $\DerCon$ is the derived
  endomorphism algebra of $T$ computed by means of the canonical DG enhancement
  of $\SingCat{R}$. The proofs of~\Cref{thm:JM-rUMP_unit} and \Cref{cor:JM-rUMP_unit} only rely on the fact
  that $T\in\SingCat{R}$ is a $2\ZZ$-cluster tilting object,
  see~\Cref{rmk:thm:JM-rUMP_unit}. Consequently, the restricted universal Massey
  product $j^*\!\class{m_4^A}$ is a unit in $\TateHH{\Lambda}[\Lambda[\imath^{\pm1}]]$
  and, by \Cref{coro:rUMP_determines}, the DG algebras $A$ and $\DerCon$ are
  quasi-isomorphic. Therefore, the DG categories
  \[
    \A\simeq\DerCat[c]{A}_{\mathrm{dg}}\qquad\text{and}\qquad\dgSingCat{R}
  \]
  are quasi-equivalent. This shows that every DG enhancement of $\SingCat{R}$ is
  equivalent to the canonical DG enhancement and the claim follows.
\end{proof}

\begin{remark}
  \Cref{coro:dgSingCat_unique} is stronger that \Cref{conj:DW}, as it shows that
  the DG category $\dgSingCat{R}$ is determined by $\Con*$ up to isomorphism in
  $\Hmo$, the Morita category of small DG categories~\cite{Tab05}.
\end{remark}

\section{Concluding remarks}

In this section we collect observations, some of which follow easily from the results in the previous sections.

\subsection{Formality of contraction algebras}

Recall that a DG algebra $A$ is \emph{formal} if there is a quasi-isomorphism
$A\simeq\dgH{A}$, where the graded algebra $\dgH{A}$ is viewed as a DG algebra
with vanishing differential. We observe that $2\ZZ$-derived contraction algebras are
almost never formal.\footnote{This fact will come as no surprise to the experts,
  but we did not find a proof of it in the literature.}

\begin{theorem}
  Let $R$ be an isolated cDV singularity with a crepant resolution and $\DerCon*$ a $2\ZZ$-derived contraction
  algebra for $R$. The following statements are equivalent:
  \begin{enumerate}
    \item\label{it:DerCon_isFormal} The $2\ZZ$-derived contraction algebra $\DerCon*$ is
          formal.
    \item\label{it:Con_isC} There is an isomorphism of algebras $\Con*\cong\CC$.
    \item\label{it:R_isAtiyahFlop} There is an isomorphism of algebras
          \[
            R\cong\CC\llbracket x,y,z,t\rrbracket/(xy-zt),
          \]
          so that $\Spec*{R}$ is the base of the Atiyah flop~\cite{Ati58}.
  \end{enumerate}
  If the above equivalent conditions hold, then there is a quasi-isomorphism
  \[
    \DerCon*\simeq\CC[\imath^{\pm1}],\qquad |\imath|=-2,
  \]
  where the graded algebra $\CC[\imath^{\pm1}]$ is equipped with the trivial
  differential.
\end{theorem}
\begin{proof}
  \eqref{it:DerCon_isFormal}$\Rightarrow$\eqref{it:Con_isC} If the $2\ZZ$-derived
  contraction algebra $\DerCon*$ is formal, then its restricted universal Massey
  product $j^*\!\class{m_4}$ is represented by the trivial sequence in
  $\HH[4][-2]{\Con*}[\Con*[\imath^{\pm1}]]=\Ext[\Con*^e]{\Con*}{\Con*}[4]$. In view of
  \Cref{thm:JM-rUMP_unit}, this can only happen if $\Con*$ is projective as a
  $\Con*$-bimodule or, equivalently, if the algebra $\Con*$ is semisimple.
  Since contraction algebras are basic and connected, we must have an
  isomorphism of algebras $\Con*\cong\CC$, which is what we needed to prove.

  \eqref{it:Con_isC}$\Rightarrow$\eqref{it:DerCon_isFormal} If $\Con*\cong\CC$,
  then there is an isomorphism of graded algebras
  \[
    \dgH{\DerCon*}\cong\Con*[\imath^{\pm1}]\cong\CC[\imath^{\pm1}],\qquad
    |\imath|=-2.
  \]
  It is well-known (and easy to prove using Kadeishvili's Theorem~\cite{Kad88},
  for example) that the latter graded algebra is intrinsically formal, that is
  every DG algebra with cohomology algebra $\CC[\imath^{\pm1}]$ is formal. In
  particular, $\DerCon*$ is formal.\footnote{Alternatively, one can prove this
    fact using the results in this note as follows: Since the enveloping algebra
    \[
      \Con*^e\cong\CC\otimes_\CC\CC\cong\CC
    \]
    is semisimple, the Hochschild--Tate cohomology
    $\TateHH{\Con*}[\Con*[\imath^{\pm1}]]$ vanishes in positive Hochschild degrees. Then,
    by~\Cref{thm:rUMP_ThmB}, the $2\ZZ$-derived contraction algebra $\DerCon*$ is
    quasi-isomorphic to its cohomology algebra $\dgH{\DerCon*}$ for the
    condition on the agreement of the corresponding restricted universal Massey products is
    trivially satisfied. Of course, this is essentially the same proof as the
    one using Kadeishvili's Theorem, which is indeed a special case
    of~\cite[Thm.~B]{JKM22}.}

  \eqref{it:Con_isC}$\Leftrightarrow$\eqref{it:R_isAtiyahFlop} In view of the
  validity of \Cref{conj:DW}, it is enough to observe that $\CC$ is indeed
  isomorphic to the contraction algebra of the Atiyah flop~\cite[Table~2]{DW16}.
\end{proof}

\begin{remark}
  The (non-)formality of the derived contraction algebra $\DerCon*^{\leq0}$ is investigated
  in~\cite[Ch.~9]{Boo19} where minimal models of this DG algebra are computed in
  various examples, see also \cite{Boo18,Boo21,Boo22}.
\end{remark}

\begin{remark}
  More generally, Kadeishvili's Theorem~\cite{Kad88} can be used to prove that the Laurent polynomial algebra
  $K[\imath^{\pm1}]=K\otimes_\CC\CC[\imath^{\pm1}]$ is intrinsically formal if $K$ is
  a finite-dimensional algebra of projective dimension at most $2$ as a
  $K$-bimodule~\cite[Cor.~4.2]{Sai23}. Such an algebra $K$, however, cannot be a contraction algebra unless $K=\CC$ since contraction algebras are basic and connected, and a symmetric $\CC$-algebra has finite projective dimension as a bimodule over itself if and only if it is separable (semisimple).
\end{remark}

\subsection{Relationship to the Derived Auslander--Iyama Correspondence}
\label{sec:der_AIC}

Let $\T$ be a triangulated category whose underlying additive category is
Krull--Schmidt and has finite-dimensional morphism spaces. A basic object
$T\in\T$ is \emph{$d$-cluster tilting}, $d\geq1$, if the following conditions
are satisfied \cite{IY08,Bel15}:
\begin{itemize}
  \item The object $T$ is \emph{$d$-rigid}: $\T(T,T[i])=0$ for all $0<i<d$.
  \item For each object $X\in\T$ there exists a diagram
        \begin{center}
          \begin{tikzcd}[column sep=tiny, row sep=small]
            &T_{d-2}\drar&\cdots&&&T_1\drar\ar{rr}&&T_0\drar[end anchor=north west]\\
            T_{d-1}\urar&&X_{d-2}\ar{ll}[description]{+1}&\cdots&X_2\urar&&X_1\urar\ar{ll}[description]{+1}&&X\ar{ll}[description]{+1}
          \end{tikzcd}
        \end{center}
\end{itemize}
in which $T_i\in\add*{T}$, $0\leq i<d$, the oriented triangles denote exact
triangles in $\T$ and the unoriented triangles commute.
The object $T$ is \emph{$d\ZZ$-cluster tilting} if it is $d$-cluster tilting
and $T\cong T[d]$.
\Cref{thm:JM-con} is a special case of the theorem below. For a
finite-dimensional algebra $\Lambda$, we let $\proj*{\Lambda}$ be the category
of finite-dimensional projective $\Lambda$-modules. For example, if $\Lambda=\T(T,T)$, then
the Yoneda functor
\[
  \T\supseteq\add*{T}\stackrel{\sim}{\longrightarrow}\proj*{\Lambda},\qquad X\longmapsto\T(T,X),
\]
is an equivalence (of additive categories).

\begin{theorem}[Derived Auslander--Iyama Correspondence
      {\cite[Thm.~5.1.10]{JKM22}}]
  \label{thm:derived_dZ-Auslander_correspondence}
  Let $d\geq1$. There is a bijective correspondence between the following:
  \begin{enumerate}
    \item\label{it:dZ-CT_DGAs} Quasi-isomorphism classes of DG algebras $A$ that
          satisfy the following:
          \begin{itemize}
            \item The algebra $\dgH[0]{A}$ is a basic finite-dimensional algebra.
            \item The free DG $A$-module $A\in\DerCat[c]{A}$ is a $d\ZZ$-cluster tilting
                  object.
          \end{itemize}
    \item\label{it:Lambda_pair} Equivalence classes of pairs $(\Lambda,I)$
          consisting of
          \begin{itemize}
            \item a basic finite-dimensional self-injective algebra $\Lambda$ and
            \item an invertible $\Lambda$-bimodule $I$ such that
                  $\Omega_{\Lambda^e}^{d+2}(\Lambda)\cong I$ in the stable category of
                  $\Lambda$-bimodules.
          \end{itemize}
  \end{enumerate}
  The correspondence is given by the formula $A\mapsto(\dgH[0]{A},\dgH[-d]{A})$.
\end{theorem}

\begin{remark}
  The case $d=1$ of \Cref{thm:derived_dZ-Auslander_correspondence} is one way
  to formulate the main result in~\cite{Mur22}. Indeed, an object $T\in\T$ is
  $1$-cluster tilting if and only if it is $1\ZZ$-cluster tilting if and only if
  $\add*{T}=\T$; the latter condition means that $\T$ is a triangulated category
  of finite type in the terminology of~\cite{Mur22}.
\end{remark}

\begin{remark}
  Let $(\Lambda,I)$ be a pair as in
  \Cref{thm:derived_dZ-Auslander_correspondence}\eqref{it:Lambda_pair}. Since
  the algebra $\Lambda$ is assumed to be basic, the map
  \[
    \Aut*{\Lambda}\longrightarrow\Pic(\Lambda),\qquad\sigma\longmapsto[\twBim{\Lambda}[\sigma]]
  \]
  from the group of algebra automorphisms of $\Lambda$ to the Picard group of
  invertible $\Lambda$-bimodules induces an isomorphism of groups~\cite[Prop.~3.8]{Bol84}
  \[
    \Out*{\Lambda}\stackrel{\sim}{\longrightarrow}\Pic(\Lambda),\qquad[\sigma]\longmapsto[\twBim{\Lambda}[\sigma]],
  \]
  where $\Out*{\Lambda}=\Aut*{\Lambda}/\operatorname{Inn}(\Lambda)$ is the group
  of outer automorphisms of $\Lambda$. In particular, there exists
  $\sigma\in\Aut(\Lambda)$ such that $I\cong\twBim{\Lambda}[\sigma]$ as
  $\Lambda$-bimodules. The condition
  $\Omega_{\Lambda^e}^{d+2}(\Lambda)\simeq\twBim{\Lambda}[\sigma]$ expresses the
  fact that the algebra $\Lambda$ is \emph{twisted $(d+2)$-periodic with respect
    to $\sigma$}. When $\sigma=\id$ or, equivalently, $I\cong\Lambda$, the
  algebra $\Lambda$ is said to be \emph{$(d+2)$-periodic}. For example,
  contraction algebras are known to be $4$-periodic. We refer the reader
  to~\cite{ES08} and the references therein for information on (twisted)
  periodic algebras.
\end{remark}

\Cref{thm:JM-con} follows from the \emph{injectivity} of the correspondence
in~\Cref{thm:derived_dZ-Auslander_correspondence} with $d=2$; the proof
of~\Cref{thm:JM-con} outlined in this note effectively goes through the proof of
the latter in the special case of $2\ZZ$-derived contraction
algebras. We record the resulting characterisation of contraction algebras for
the sake of completeness.

\begin{theorem}
  \label{thm:cDV-2Z-CT}
  Let $R$ be an isolated cDV singularity with a crepant resolution $p\colon
    X\to\Spec*{R}$. Up to quasi-isomorphism, the $2\ZZ$-derived contraction algebra
  $\DerCon=\DerCon*(p)$ is the unique DG algebra with the following properties:
  \begin{enumerate}
    \item $\DerCon\in\DerCat[c]{\DerCon}$ is a $2\ZZ$-cluster tilting object.
    \item There is an isomorphism of algebras
          $\dgH[0]{\DerCon}\cong\Con=\Con*(p)$.
    \item There is an isomorphism of $\Con$-bimodules $\dgH[-2]{\DerCon}\cong\Con$.
  \end{enumerate}
  In other words, $\DerCon$ is determined up to quasi-isomorphism by its image
  $(\Con,\Con)$ under the Derived Auslander--Iyama Correspondence.
\end{theorem}
\begin{proof}
  That the $2\ZZ$-derived contraction algebra satisfies the first two properties
  follows from~\Cref{thm:Wemyss} since, by definition, $\DerCon$ is the derived
  endomorphism algebra of a $2\ZZ$-cluster tilting object in $\SingCat{R}\simeq\DerCat[c]{\DerCon}$. 
  The third property was checked in the computation of the cohomology of the $2\ZZ$-derived contraction algebra.
  Therefore $\DerCon$ belongs to the
  class of DG algebras in
  \Cref{thm:derived_dZ-Auslander_correspondence}\eqref{it:dZ-CT_DGAs} with $d=2$.
  Moreover, $\DerCon$ is determined up to quasi-isomorphism by its image $(\Con,\Con)$
  under the Derived Auslander--Iyama Correspondence, which is what we needed to
  prove.
\end{proof}

\subsection{Isolated cDV singularities with non-smooth minimal models}

Contraction algebras are defined for an arbitrary cDV singularity $R$ that is
neither isolated nor admits a crepant resolution. However, contraction algebras
are finite-dimensional if and only if $R$ defines an isolated
singularity~\cite[Summary~5.6]{DW19a}, and this finite-dimensionality is crucial
to our approach. On the other hand, $R$ admits a crepant resolution if and only
if the singularity category $\SingCat{R}$ admits a $2\ZZ$-cluster tilting
object, see~\cite[Thm.~5.4]{BIKR08} and the references therein. If, on the other
hand, the minimal models\footnote{In the context of the MMP in dimension three
  and higher, minimal models play the role of minimal resolutions of surfaces.
  We do not recall the technical definition in this note, but only mention that
  minimal models are permitted to have `mild' singularities as long as they remain
  `closer' to the original space than a smooth resolution (which always exists by a
  famous theorem of Hironaka~\cite{Hir64}.)} of $R$ are singular, then the contraction algebras of
$R$ are the endomorphism algebras of \emph{maximal rigid objects} in
$\SingCat{R}$~\cite{Wem18}, that is objects $T\in\SingCat{R}$ such that
\[
  \add*{T}=\set{X\in\SingCat{R}}[\Hom{T\oplus X}{(T\oplus X)[1]}=0].
\]
It is easy to verify that $2$-cluster tilting objects are maximal rigid, but the
converse is false in general~\cite{BIKR08,BMV10}; moreover, if there exists a
$2$-cluster tilting object then every maximal rigid object is also $2$-cluster
tilting~\cite[Thm.~II.1.8]{BIRS09}.\footnote{The reader should compare this
  statement with the following geometric fact: If one minimal model of
  $\Spec*{R}$ is smooth, then all of its minimal models are also
  smooth~\cite[Cor.~4.11]{Kol89}.} In any case, one may still define the $2\ZZ$-derived contraction
algebras of $R$ as the derived endomorphism algebras of maximal rigid objects in
$\SingCat{R}$, computed in terms of the canonical DG enhancement of the latter
triangulated category. Note, however, that
\Cref{thm:derived_dZ-Auslander_correspondence} does not cover the case of
maximal rigid objects that are not $2\ZZ$-cluster tilting and hence it cannot be
applied to prove a version of \Cref{thm:cDV-2Z-CT} for isolated cDV
singularities that do not admit a crepant resolution. Finally, we mention that
the apparent variant of \Cref{conj:DW} does not hold for isolated cDV
singularities whose minimal models are not smooth, see \cite[Ex~8.4.2]{Boo21}
for an explicit counterexample.

\bibliographystyle{halpha}%

\bibliography{mylibrary}

\end{document}